\documentclass[12pt]{article}
\usepackage{amsmath,amsfonts,amsthm,amscd,amssymb,latexsym,comment}
\usepackage{enumerate,graphicx,psfrag,subfigure,changebar,oldgerm}
\usepackage[mathscr]{eucal}
\title{Orthogonal systems in finite graphs}
\author{ \textsf{Andrew J. Duncan}
 \and \textsf{Ilya V. Kazachkov}
 \and \textsf{Vladimir N. Remeslennikov}}

\def\nul{\emptyset }
\def\D{\Delta }

\def\b{\beta }

\def\i{\iota }

\def\G{\Gamma }
\def\g{\gamma }
\def\a{\alpha }
\def\t{\tau }

\def\s{\sigma }

\def\fC{{\textswab C}}
\newtheorem{thm}{Theorem}[section]
\newtheorem{lem}[thm]{Lemma}
\newtheorem{cor}[thm]{Corollary}
\newtheorem{prop}[thm]{Proposition}

\newtheorem{defn}[thm]{Definition}
\newtheorem*{defn*}{Definition}
\newtheorem{exam}[thm]{Example}
\newenvironment{expl}{\begin{exam} \rm}{\end{exam}}
\newtheorem{rem}[thm]{Remark}

\numberwithin{equation}{section}
\numberwithin{figure}{section}

\newcommand{\tc}{\texttt{c}}
\newcommand{\Aut}{\operatorname{Aut}}

\newcommand{\cd}{\operatorname{cdim}}

\newcommand{\NN}{\ensuremath{\mathbb{N}}}

\newcommand{\cC}{\mathcal{C}}
\newcommand{\cO}{\mathcal{O}}

\newcommand{\la}{\langle}
\newcommand{\ra}{\rangle}

\newcommand{\cl}{\operatorname{cl}}
\newcommand{\icl}{\operatorname{icl}}
\newcommand{\acl}{\operatorname{acl}}
\newcommand{\fcl}{\operatorname{fcl}}
\newcommand{\CS}{
{L}}
\newcommand{\maps}{\rightarrow}
\newcommand{\ov}[1]{\overline{#1}}

\newcommand{\be}{\begin{enumerate}}
\newcommand{\ee}{\end{enumerate}}
\newcommand{\id}{\operatorname{id}}
\newcommand{\ilya}[1]{\marginpar{{\LARGE $\mathbf{\Downarrow IVK}$}}\textbf{ #1}}
\newlength{\nts}
\newlength{\rts}
\newlength{\lts}
\begin{document}

\maketitle

\begin{abstract}
Let $\Gamma$ be a finite graph and $G_\Gamma$ be the corresponding
free partially commutative group. In this paper we
construct orthogonality theory for graphs and free partially
commutative groups. The theory developed here provides tools for the
study of the structure of the centraliser  lattice of partially commutative
groups.
\end{abstract}
 \tableofcontents
%
\subsection*{Glossary of Notation}
\begin{tabbing}
$\G$ \hspace*{\lts}\= ---  \hspace*{\rts}\=\parbox{9.5cm}{a finite undirected graph with vertex set $X$}\\
$\G_1\oplus \G_2$ \> --- \> the join of graphs $\G_1$ and $\G_2$
\\[\nts]
  $C_G(S)$ \> --- \> \parbox[t]{9.5cm}{the centraliser of a a subset $S$ of $G$} \\[\nts]
  $\fC(G)$ \> --- \> \parbox[t]{9.5cm}{the set of centralisers of a group $G$} \\[\nts]
  $G$ or $G(\G)$ \> --- \> \parbox[t]{9.5cm}{the (free) partially commutative group with underlying graph $\G$}\\[\nts]
  $\lg(w)$ \> --- \> \parbox[t]{9.5cm}{the length of a geodesic word $w'$ such that $w =_{G} w'$}\\[\nts]
  $d(x,y)$ \> --- \> \parbox[t]{9.5cm}{the distance from $x$ to $y$, $x,y \in \G$}\\[\nts]
  $\cO^Z(Y)$ \> --- \> \parbox[t]{9.5cm}{the orthogonal complement of $Y$ in 
$Z$, i.e.  $\{u\in Z|d(u,y)\le 1, \textrm{ for all } y\in Y\}$}\\[\nts]
  $Y^\perp$ \> --- \> \parbox[t]{9.5cm}{the orthogonal complement of $Y$ in $X$,
  $\cO^X(Y)$} \\[\nts]
  $\cl^Z(Y)$ \> --- \> \parbox[t]{9.5cm}{the closure of $Y$ in $Z$ with
  respect to $\cO^Z(Y)$, i.e. $\cl^Z(Y)=\cO^Z(\cO^Z(Y))$} \\[\nts]
  $\cl(Y)$ \> --- \> \parbox[t]{9.5cm}{the closure of $Y$ in $X$, i.e. $\cl(Y)=Y^{\perp \perp}$}\\[\nts]
  $\CS(\G)$ or $L$ \> --- \> \parbox[t]{9.5cm}{the lattice of closed sets of $\G$}\\[\nts]
  $\ov{X}$ \> --- \> \parbox[t]{9.5cm}{the set $X\cup\{t\}$}\\[\nts]
  $\ov\G$ \> --- \> \parbox[t]{9.5cm}{the graph $(\ov X, E(\G)\cup E_t)$, $E_t=\{(t,x)|x\in J_t\}$, $J_t\subseteq X$}\\[\nts]
  $\ov L$ \> --- \> \parbox[t]{9.5cm}{the lattice of closed sets, 
$\CS(\ov \G)$, of $\ov \G$ }\\[\nts]
  $ L_t$\> --- \> \parbox[t]{9.5cm}{$\{Y\subseteq X\mid Y=C\cap J_t, \textrm{ where } C\in
  L\}$}\\[\nts]
$ \tilde{L}$\> --- \> \parbox[t]{9.5cm}{the set $L\cup L_t$}\\[\nts]
  $h(L)$  \> --- \> \parbox[t]{9.5cm}{the height of a lattice $L$}\\[\nts]
  $Y\sim_\perp Z$  \> --- \> \parbox[t]{9.5cm}{$Y,Z\subseteq X$ are $\perp$-equivalent in
  $X$, that is $Y^\perp=Z^\perp$}\\[\nts]
  $\acl(S)$  \> --- \> \parbox[t]{9.5cm}{the Abelian closure of a simplex $S$, that is the union of $T\subseteq X$ such that $S\sim_\perp T$}\\[\nts]
  $Y\sim_o Z$ \> --- \> \parbox[t]{9.5cm}{subsets $Y,Z\subseteq X$ are $o$-equivalent,
  i.e. $Y^\perp \smallsetminus Y =Z^\perp \smallsetminus Z$}\\[\nts]
  $\fcl(A)$ \> --- \> \parbox[t]{9.5cm}{the free-closure of a free co-simplex $A$, that is the union
of all free co-simplexes $B$ such that $A\sim_o B$}\\[\nts]
  $[x]_\perp$ \> --- \> \parbox[t]{9.5cm}{the $\perp$-equivalence class
  of $x$, that is $\{y\in X\ \mid \ x\sim_\perp y\}$}\\[\nts]
  $[x]_o$ \> --- \> \parbox[t]{9.5cm}{the $o$-equivalence class
  of $x$, that is $\{y\in X\ \mid \ x\sim_o y\}$}\\[\nts]
  $x \sim y$ \> --- \> \parbox[t]{9.5cm}{ $x,y \in X$ are equivalent,
  i.e. either $x\sim_\perp y$ or $x\sim_o y$}\\[\nts]
  $[x]$ \> --- \> \parbox[t]{9.5cm}{the equivalence class of $x$ with
  respect to $\sim$}\\[\nts]
  $\G^{\tc}$ \> --- \> \parbox[t]{9.5cm}{the compression of the graph
  $\G$}
\end{tabbing}
\section*{Introduction}
This paper is a continuation of  a series of papers \cite{DKR1,DKR2}
where the authors develop the theory of free partially commutative
groups.

Free partially commutative groups arise in many branches of
mathematics and computer science and consequently are known
by a variety of names: \emph{semifree groups}, \emph{graph
groups}, \emph{right-angled Artin groups}, \emph{trace groups},
\emph{locally free groups}. We refer the reader to \cite{Diekert},
\cite{EKR} and references there for a survey of these groups, which we
shall refer to here as partially
commutative groups.

The analysis of proofs of results on partially commutative
groups shows that these rely heavily upon two main ideas: divisibility and
orthogonality. The divisibility theory of  partially commutative
groups has been formalised in \cite{EKR} and is a convenient tool
for solving major algorithmic problems. The idea of considering
orthogonal complements of subsets of vertices of the underlying
graph of a  partially commutative
implicitly occurs in many papers, see for instance,
\cite{Servatius, Laurence} and  also  \cite{HBIG} pp. 650-651. In
this paper we formalise this idea and establish the main results of
orthogonality theory for graphs.

\begin{defn*}
Let $G(\G$) be the partially commutative group with underlying
graph $\Gamma=(X,E)$. For a vertex $x\in X$ we define $x^\perp$ to
be the set of all vertices of $\Gamma$ connected with $x$. For a
subset $Y\subseteq X$ we define
\[Y^\perp=\bigcap\limits_{y\in Y} y^\perp.\]
Let $\CS(\Gamma)$ be the set of all subsets $Z$ of $X$ of the form
$Y^\perp$ for some $Y\subseteq X$. We call $\CS(\Gamma)$ the
lattice of closed sets of $\Gamma$.
\end{defn*}

The importance of the lattice of closed sets $L(\Gamma)$ for the
theory of partially commutative groups is a consequence of the
the fact that the lattice $\CS(\Gamma)$ is isomorphic to the lattice
of parabolic centralisers (see Section \ref{sec:prelim}) of $G(\G)$ which, in
turn, is crucial for study of the group $G(\G)$ itself and its
automorphism group $\Aut(G(\G))$.

The main problem that we consider in this paper is how the lattice
of closed set behaves when one joins a vertex $v$ to the graph
$\Gamma$ to form a new graph $\bar \G$.
Naturally this depends on which vertices of $\G$
are joined to $v$.  In particular, we prove that the lattices $\CS=\CS(\G)$
and $\ov \CS=\ov{\CS(\G)}$ are
isomorphic if and only if $v$ is joined to the orthogonal complement
of a simplex $S\subset X$; see Theorem \ref{thm:ginj}.

Moreover, we prove that the height $h(\ov L)$ of the extended
lattice $\ov L$  is $h(\ov L)=h(L)+m$, where $m=0,1$ or
$2$, see Theorem \ref{thm:hL}.

In Sections \ref{sec:abinf} and \ref{sec:FI} we introduce operations
of free and Abelian inflation and deflation on graphs and prove that
the lattice of closed sets $L$ behaves nicely under these operations.
We then introduce the notion of compression of a graph $\Gamma$
which plays an important role in the study of  partially
commutative groups and prove that the lattices of closed sets for
the graph $\Gamma$ and its compression are closely related.
The compression of a graph allows us to give a decomposition of
the automorphism group of the graph as a semi-direct product of the
automorphism group of the compression with a direct sum of symmetric groups.

The results of the current paper play a key role in two papers of
authors which are currently under preparation: one on the structure
of lattices of centralisers of a given partially commutative
group $G$, the other on the structure of the automorphism group
$\Aut(G)$, \cite{DKR4,DKR5}.

A major part of our research on partially commutative groups,
\cite{DKR1,DKR2,DKR3,DKR4,DKR5} was carried out while the second and
the third authors were visiting the University of
Newcastle Upon Tyne, thanks to the support of the EPSRC  grants 
EP/D065275/1 and GR/S61900/01.

\section{Preliminaries}\label{sec:prelim}
\subsection{Graphs}\label{sec:graph}
Graph will mean undirected, finite graph throughout this paper.
%
If $x$ and $y$ are vertices of a  graph then we
define the {\em distance} $d(x,y)$ from $x$ to
$y$ to be the minimum of the lengths of all paths from $x$ to $y$ in $\G$.
A subgraph $S$ of a graph $\G$ is called a {\em full} subgraph if vertices $a$ and $b$ of
$S$ are joined by an edge of $S$ whenever they are joined by an edge of $\G$.

Let $\G$ be a graph with $V(\G)=X$.
A subset $Y$ of $X$ is called a {\em simplex}
if the full subgraph of $\G$ with vertices $Y$ is isomorphic to  a
complete graph. A maximal simplex is called a {\em clique}.
A subset $Y$ of $X$ is called a {\em free co-simplex}
if the full subgraph of $\G$ with vertices $Y$ is isomorphic to  the
null graph. The reason why the word ``free'' is necessary here will
become apparent later (see Section \ref{sec:FI}).

Let $\G_i$ be a graph with vertex set $X_i$, for $i=1,2$.
The {\em join} $\G_1\oplus \G_2$ of $\G_1$ and $\G_2$ is the graph with vertex
set the disjoint union $X_1\sqcup X_2$ and edge set consisting
of all the edges of $\G_i$, for $i=1$ and $2$ and an edge joining
$x_1$ to $x_2$ for all $x_1\in X_1$ and $x_2\in X_2$.

\subsection{Lattices}\label{sec:lattice}
Let $P$ be a partially ordered set with order relation $\le$.
Then $P$ is said to be a {\em lattice} if every pair of elements
of $P$ has a unique infimum and a unique supremum.
We usually write $s\wedge t$ and $s\vee t$ for the infimum and supremum,
respectively, of $s$ and $t$.

A lattice is said to be {\em bounded} if it has both a minimum and a maximum
element. An {\em ascending chain} in a lattice is  a sequence of elements
$a_0,a_1,\ldots $ such that $a_i<a_{i+1}$. The {\em length} of a finite
chain $a_0<\cdots <a_k$ is said to be $k$. {\em Descending chains } are defined
analogously. A lattice may be bounded and have infinite ascending or descending
chains (or both). The {\em height} of a lattice $L$ is defined to be the maximum
of the lengths of all chains  in $L$, if it exists, and $\infty$ otherwise.

A {\em homomorphism of partially ordered sets}
is a map from one partially ordered set to another which preserves the order relation.
If $P$ and $Q$ are lattices then a homomorphism of partially ordered sets
$f:P\maps Q$ is called a {\em homomorphism of lattices} if $f(s\vee t)=f(s)\vee f(t)$ and
$s\wedge t= f(s)\wedge f(t)$, for all $s,t\in P$.
For further details on lattices we refer the reader to
\cite{Bi}.

\subsection{Centraliser Lattices}\label{sec:centraliser}
If $S$ is a subset of a group $G$ then the centraliser
of $S$ in $G$ is
$C_G(S)=\{g\in G: gs=sg, \textrm{ for all } s\in S\}$. We write
$C(S)$ instead of $C_G(S)$ when the meaning is clear.
Let $\fC(G)$ denote the set of centralisers of a group $G$.
The relation of inclusion then defines  a partial order `$\le$' on $\fC(G)$.
We define the infimum of a pair of elements of $\fC(G)$ in the
obvious way as:
    \[
        C(M_1) \wedge C(M_2)= C(M_1)\cap C(M_2)=C( M_1 \cup M_2).
    \]
Moreover  the supremum $C(M_1) \vee C(M_2)$ of elements
$C(M_1)$ and $C(M_2)$ of $\fC(G)$ may
be
defined to be the intersection of all centralisers containing $C(M_1)$ and $C(M_2)$.
Then  $C(M_1) \vee C(M_2)$ is  minimal among centralisers containing $C(M_1)$ and $C(M_2)$.
These definitions make $\fC(G)$ into a lattice, called the
{\em centraliser lattice} of $G$. This lattice is bounded as it has
a greatest element, $G= C(1)$, and a least element, $Z(G)$, the centre
of $G$. Lattices of centralisers have been extensively studied;
a brief survey of results
can be found in \cite{DKR1}.

The {\em centraliser dimension} of a group $G$ is defined to be the height of the centraliser lattice of $G$  and is denoted $\cd(G)$.
Centralisers have the properties that, for all subsets $S$ and $T$ of $G$,
if $S\subseteq T$ then $C(S)\ge C(T)$ and $C(C(C(S)))=C(S)$.
Therefore if $C_1<C_2<\cdots $ is an ascending chain then
$\cdots >C(C_2)>C(C_1)$ is a descending chain and both chains are either
infinite or of the same length.
Thus $\cd(G)$ is the
maximum of the lengths of descending chains of centralisers in $G$, if such
a maximum exists, and is infinite otherwise.
\subsection{Partially Commutative Groups}
Let $\G$ be a finite, undirected, simple graph. Let $X=V(\G)$ be the set of vertices
of $\G$ and let $F(X)$ be the free group on $X$.
For elements $g,h$ of a group we denote the commutator $g^{-1}h^{-1}gh$
of $g$ and $h$ by $[g,h]$. Let
\[
R=\{[x_i,x_j]\in F(X)\mid x_i,x_j\in X \textrm{ and there is an edge from } x_i \textrm{ to } x_j
\textrm{ in } \G\}.
\]
We define
the {\em partially commutative group with {\rm(}commutation\rm{)} graph } $\G$ to be the
group $G(\G)$ with presentation
$
\left< X\mid R\right>.
$
(Note that these are the groups which are
called finitely generated free partially commutative groups
in \cite{DK}.)

Let $\G$ be a simple graph, $G=G(\G)$ and
let $w\in G$. Denote by $\lg(w)$ the length of a {\em geodesic} word
in $X\cup X^{-1}$ that represents the element $w\in G$: that is  a
word of minimal length amongst those representing $w$. We say that   $w \in G$
 is \emph{cyclically minimal} if and only if
$$
\lg(g^{-1}wg) \ge \lg(w)
$$
for every $g \in G$.

The centraliser dimension of partially commutative groups
is finite because
all partially commutative groups are linear \cite{H}
 and all linear groups have finite 
centraliser dimension, \cite{MS}.
In \cite{DKR2} it is shown that the centraliser dimension of a
partially commutative group is easy to calculate and
depends only on the centralisers of subsets of $X$. If $Y\subseteq X$ then
we call $C(Y)$ a {\em canonical parabolic} centraliser.
It is not hard to prove that
the intersection of two canonical parabolic centralisers is again a
canonical parabolic centraliser and,
as shown in \cite{DKR2}, the supremum, in $\fC(G)$,
of two canonical parabolic centralisers
is also a canonical parabolic centraliser. Hence the set $\fC(X;G)$ of
canonical parabolic centralisers
forms a sublattice of $\fC(G)$. In \cite[Theorem 3.3]{DKR2} it is shown that
the centraliser dimension of $G$ is equal to the height of the lattice  $\fC(X;G)$. In \cite{DKR3} we give a short proof of this fact using the
methods developed in this paper and give a characterisation
of centralisers of arbitrary subsets of a partially commutative group.
Moreover in \cite{DKR4, DKR5} we use these tools to give a description
of the automorphism group of a partially commutative group.

\section{The Lattice of Closed Subsets of a  Graph} 
\label{sec:2}

\subsection{Orthogonal Systems, Closure and Closed Sets}\label{sec:OC}
As before let $\G$ be a finite, undirected, simple graph, with vertices
$X$, and let $G=G(\G)$ be the partially commutative group
defined by $\Gamma$. Given vertices $x,y$ in the same
connected component of $\G$ we define the
{\em distance} $d(x,y)$ from $x$ to $y$ to be the minimum of
the lengths of paths from $x$ to $y$. If $x$ and $y$ are
in distinct connected components then we define $d(x,y)=\infty$.

Let $Y$ and $Z$ be subsets of $X$. We define the {\em orthogonal complement}
of $Y$ in $Z$ to be
\[\cO^Z(Y)=\{u\in Z|d(u,y)\le 1, \textrm{ for all } y\in Y\}.\]
By convention we set $\cO^Z(\nul)=Z$.
If $Z=X$ we call $\cO^X(Y)$ the orthogonal complement of $Y$, and if
no ambiguity arises then we shall sometimes write $Y^\perp$ instead
of $\cO^X(Y)$. Also, if every vertex of $Z$ is either in $Y$ or is joined
 by an edge of $\G$ to
every vertex of $Y$ then we write $[Y,Z]=1$. Thus $[Y,Z]=1$ if and only
if $Z\subseteq \cO^X(Y)$ if and only if every element of $Y$ commutes with
every element of $Z$ in the group $G$. For future reference we record some
of the basic properties of orthogonal complements in the next lemma.
\begin{lem}\label{lem:orth} Let $Y$, $Y_1$, $Y_2$ and
$Z$ be subsets of $X$.
\begin{enumerate}
\item\label{it:orth1} If $Y\subseteq Z$ then $Y\subseteq \cO^Z(\cO^Z(Y))$.
\item\label{it:orth2} If $Y\subseteq Z$ then $\cO^Z(Y)=
    \cO^Z(\cO^Z(\cO^Z(Y)))$.
\item\label{it:orth3} If $Y_1\subseteq Y_2$ then $\cO^Z(Y_2) \subseteq
        \cO^Z(Y_1)$.
\item\label{it:orth4}   $\cO^Z(Y_1\cap Y_2)\supseteq \cO^Z(Y_1)
        \cup \cO^Z(Y_2)$.
\item\label{it:orth5}  
    $\cO^Z(Y_1\cup Y_2)=\cO^Z(Y_1)
        \cap \cO^Z(Y_2)$.
\item\label{it:orth6} $Y$ is a simplex if and only if $Y\subseteq Y^\perp$.
\item\label{it:orth7} $Y$ is a clique if and only if $Y=Y^\perp$.
\end{enumerate}
In particular from \ref{it:orth1} and \ref{it:orth2} we have
$Y\subseteq Y^{\perp\perp}$ and $Y^\perp=Y^{\perp\perp\perp}$,
where we write $Y^{\perp\perp}$ for $(Y^\perp)^\perp$.
\end{lem}
\begin{proof}
If $y\in Y\subseteq Z$ then, for all $u\in \cO^Z(Y)$, we have $d(u,y)\le 1$.
Hence
$y\in \cO^Z(\cO^Z(Y))$ and \ref{it:orth1} follows.
Statement \ref{it:orth3} follows directly from the definition of
orthogonal complement. Statement \ref{it:orth2} follows from
\ref{it:orth1} and \ref{it:orth3}. Statement \ref{it:orth4} follows
from \ref{it:orth3}. To see \ref{it:orth5} suppose first that $Z=X$.
It follows
from \ref{it:orth3} that
$\cO^X(Y_1\cup Y_2)\subseteq \cO^X(Y_1)
        \cap \cO^X(Y_2)$.
From \ref{it:orth4} and \ref{it:orth1} we have
$\cO^X(\cO^X(Y_1)\cap \cO^X(Y_2))\supseteq Y_1\cup Y_2$.
Hence, from \ref{it:orth1} and \ref{it:orth3},
$\cO^X(Y_1)\cap \cO^X(Y_2)\subseteq
\cO^X(\cO^X(\cO^X(Y_1)\cap \cO^X(Y_2)))
\subseteq \cO^X(Y_1\cup Y_2)$, so \ref{it:orth5} holds in this case. In general,
$\cO^Z(Y_1\cup Y_2)=\cO^X(Y_1\cup Y_2)\cap Z = \cO^X(Y_1)\cap \cO^X(Y_2)
\cap Z= \cO^Z(Y_1)\cap \cO^Z(Y_2)$, as required.
If $Y$ is a simplex and $y\in Y$ then $d(y,z)=1$, for all $z\in Y$,
$z\neq y$. Hence $Y\subset Y^\perp$. Conversely, if $Y\subseteq
Y^\perp$ and $y,z\in Y$ then $d(y,z)\le 1$, so $Y$ is a simplex. Therefore
\ref{it:orth6} holds. If $Y$ is a clique and $x\in Y^\perp\backslash Y$ then
$Y\cup \{x\}$ is a simplex, contrary to maximality of $Y$. Hence, using \ref{it:orth6},
$Y=Y^\perp$. Conversely, if $Y=Y^\perp$, then $Y$ is a simplex and,
by a similar argument,
there is no simplex strictly containing $Y$. Hence \ref{it:orth7} holds.
\end{proof}
\begin{expl} \label{ex:perp}
    \begin{enumerate}
    \item In general the inclusions of Lemma \ref{lem:orth} are strict.
For instance, take
$\G$ to be the graph of Figure \ref{fig:path3}, let $Y_1=\{a,c\}$ and
$Y_2=\{b,c,d\}$. Then $Y_1^\perp = \{b\}$, $Y^\perp_2=\{c\}$ and
$(Y_1\cap Y_2)^\perp=\{b,c,d\}$: so $(Y_1\cap Y_2)^\perp\neq Y_1^\perp
\cup Y^\perp_2$. Moreover $Y_1^{\perp\perp}=\{a,b,c\}\neq Y_1$.
    \item   The subgroup $G(X^{\perp})$ is the centre of the
    group $G=G(\G)$.
    \item   If $X=X_1\sqcup X_2$ is a disjoint union of $X_1$ and
    $X_2$ and $\G$ is the direct sum of graphs $\G(X_1)$ and $\G(X_2)$ then
    $G= G(X_1)\times G(X_2)$. If $\cO^{X_1}(X_1)=\cO^{X_2}(X_2)=\nul$ then the
    groups $G(X_i)$, $i=1,2$ have trivial centre.
    In this case $\cO^X(X_1)=X_2$ and $\cO^{X}(X_2)=X_1$.
    \end{enumerate}
\end{expl}
\begin{figure}
\psfrag{a}{$a$}
\psfrag{b}{$b$}
\psfrag{c}{$c$}
\psfrag{d}{$d$}
\begin{center}
\includegraphics[scale=0.4]{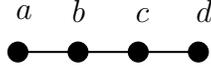}
\end{center}
\caption{A path graph}\label{fig:path3}
\end{figure}
The connection between orthogonal complements and centralisers is
made explicit in the following lemma.
\begin{lem}
Let $G=G(\G)$ and $Y\subseteq X$. Then $C_G(Y)=G(Y^\perp)$.
\end{lem}
\begin{proof}
If $x\in X$ then
$C_G(x)\supseteq G(x^\perp)$. From
\cite[Lemma 2.4]{EKR} we also have $C_G(x)\subseteq G(x^\perp)$.
Hence $C_G(Y)=\cap_{y\in Y} C_G(y)=\cap_{y\in Y}G(y^\perp)
=G(\cap_{y\in Y} y^\perp)=G(Y^\perp)$.
\end{proof}

For subsets $Y$ and $Z$ of $X$ we define the {\em closure} of $Y$ in $Z$ to be
$\cl^Z(Y)=\cO^Z\cO^Z(Y)$. When $Z=X$ 
we write $\cl(Y)$ for $\cl^X(Y)$.
The closure operator in $\G$ satisfies the following properties.
\begin{lem}\label{lem:cl} Let $Y$, $Y_1$, $Y_2$ and $Z$ be subsets of $X$.
\begin{enumerate}
\item\label{it:cl1} $Y\subseteq \cl(Y)$.
\item\label{it:cl2} $\cl(Y^\perp)=Y^\perp$.
\item\label{it:cl3} $\cl(\cl(Y))=\cl(Y)$.
\item\label{it:cl4} If $Y_1\subseteq Y_2$ then $\cl(Y_1)\subseteq \cl(Y_2)$.
\item\label{it:cl5} $\cl(Y_1\cap Y_2)\subseteq \cl(Y_1)\cap \cl(Y_2)$
and $\cl(Y_1)\cup \cl(Y_2) \subseteq\cl(Y_1\cup Y_2)$.
\item\label{it:cl6} If $Z=\cl(Y)$ then $Z=U^\perp$, for some $U\subset X$, and then
$\cl(U)=Z^\perp=Y^\perp$.
\item\label{it:cl7} If $\cl(Y_1)=\cl(Y_2)$ then $Y_1^\perp=Y_2^\perp$.
\item\label{it:cl8} $Y$ is a simplex if and only if $\cl(Y)$ is a
simplex if and only if $\cl(Y)\subseteq Y^\perp$.
\item\label{it:cl9} If $Y_1\subseteq Y_2$ then $\cl(\cl(Y_1)\cap Y_2)=\cl(Y_1)$.
\item\label{it:cl10} $\cl(\cl(Y_1)\cup \cl(Y_2))=\cl(Y_1\cup Y_2)$.
\item\label{it:cl11} $\cl(\cl(Y_1)\cap Y_2)\cap Y_2=\cl(Y_1)\cap Y_2$.
\end{enumerate}
\end{lem}
\begin{proof} Statements
\ref{it:cl1} and \ref{it:cl2} are restatements of
Lemma \ref{lem:orth}, \ref{it:orth1}
and
\ref{it:orth2}, respectively.
To see \ref{it:cl3} apply the operator $\cO^X$ to both sides
of \ref{it:cl2}.
Statement \ref{it:cl4} is a consequence of
Lemma \ref{lem:orth}.\ref{it:orth3}.
Statement \ref{it:cl5} follows from
\ref{it:cl4}.
If $Z=\cl(Y)$ then $Z=U^\perp$, where $U=Y^\perp$. If $Z=U^\perp$ then
$\cl(U)=U^{\perp\perp}=Z^\perp=(\cl(Y))^\perp=Y^\perp$, using Lemma
\ref{lem:orth}.\ref{it:orth2}. Hence \ref{it:cl6} holds. To see
\ref{it:cl7}  apply the operator $\cO^X$ to
both $\cl(Y_1)$ and
$\cl(Y_2)$ and use Lemma \ref{lem:orth}.\ref{it:orth2}.
For \ref{it:cl8}, if $\cl(Y)$ is a simplex then $\cl(Y)\subseteq \cl(Y)^\perp$,
from Lemma \ref{lem:orth}.\ref{it:orth6}, so from \ref{it:cl1} and 
Lemma \ref{lem:orth}.\ref{it:orth2}
$\cl(Y)\subseteq Y^{\perp}$. 
If $\cl(Y)\subseteq Y^\perp$ then, 
from \ref{it:cl1} and Lemma \ref{lem:orth}.\ref{it:orth6},  
$Y\subseteq Y^\perp$ so $Y$ is a simplex,
 and 
$\cl(Y)\subseteq \cl(Y^\perp)=\cl(Y)^\perp$, so $\cl(Y)$ is a simplex.  
$Y^{\perp\perp}\subseteq (Y^{\perp\perp})^\perp=Y^\perp$ so
$Y\subseteq Y^{\perp\perp}\subseteq Y^\perp$; and $Y$ is a simplex.
In the setting of \ref{it:cl9} we have,
from \ref{it:cl1}, $Y_1\subseteq \cl(Y_1)\cap Y_2$, so
$\cl(Y_1)\subseteq \cl(\cl(Y_1)\cap Y_2)$.
On the other hand
$\cl(Y_1)\cap Y_2\subseteq \cl(Y_1)$ so,
from \ref{it:cl3} and \ref{it:cl4}, $\cl(\cl(Y_1)\cap Y_2)\subseteq \cl(Y_1)$.
To see  \ref{it:cl10} use the second part of \ref{it:cl5} and then \ref{it:cl3} to
obtain $\cl(\cl(Y_1)\cup \cl(Y_2))\subseteq \cl(Y_1\cup Y_2)$. For the opposite
inclusion use \ref{it:cl1} to obtain $Y_1\cup Y_2\subseteq \cl(Y_1)\cup \cl(Y_2)$
and then \ref{it:cl4} implies that
$\cl(Y_1\cup Y_2)\subseteq \cl(\cl(Y_1)\cup \cl(Y_2))$, as required.
For \ref{it:cl11} first note that \ref{it:cl1} implies that
$\cl(Y_1)\cap Y_2\subseteq \cl(\cl(Y_1)\cap Y_2)\cap Y_2$. Also
$\cl(Y_1)\cap Y_2\subseteq \cl(Y_1)$ so \ref{it:cl4} and \ref{it:cl3} imply
that $\cl(\cl(Y_1)\cap Y_2)\subseteq \cl(Y_1)$. On intersection with $Y_2$ this
gives the inclusion required to complete the proof.
\end{proof}

\begin{expl} \label{ex:cl}
    \begin{enumerate}
    \item\label{ex:cl3}   If $x\in X$ and $Y=\cl(x)=x^{\perp \perp}$ then
    $Y$ is a simplex.
\item In terms of the group $G$ the subset $Y$ of $X$ is a simplex
if and only if $G(Y)$ is Abelian. As
 $C_G(Z)=G(Z^\perp)$, for any subset $Z$ of $X$, 
Lemma \ref{lem:orth}.\ref{it:orth6} states that  $G(Y)$ is
Abelian  if and only if $G(Y)\subseteq C_G(Y)$. The content of
Lemma \ref{lem:cl}.\ref{it:cl8} is that  $G(Y)$ is Abelian if and only if
$C^2_G(Y)$ is Abelian if and only if $C^2_G(Y) \subseteq C_G(Y)$.
    \end{enumerate}
\end{expl}
\begin{defn}
A subset $Y$ of $X$ is called {\em closed} {\rm(}with respect to $\G${\rm)}
if $Y=\cl(Y)$.
Denote by $\CS(\G)$ the set of all closed subsets of $X$.
\end{defn}

We list some basic properties of $\CS(\G)$.
\begin{lem}\label{lem:cs}
Let $Y$ be a subset of $X$. The following hold.
\be
\item $\cl(Y)\in \CS(\G)$.
\item $X$ is the unique maximal element of $\CS(\G)$.
\item $Y$ is closed if and only if $Y=\cO^X(U)$, for some $U\in \CS(\G)$.
\item $\cO^X(X)$ is the unique minimal element of  $\CS(\G)$.
\item If $Y_1, Y_2\in \CS(\G)$ then $Y_1\cap Y_2\in \CS(\G)$.
\ee
\end{lem}
\begin{proof}
\be
\item This follows from Lemma \ref{lem:cl}.\ref{it:cl3}.
\item This is clear, given the previous statement and the fact that
$X\subseteq \cl(X)$.
\item It follows, from Lemma \ref{lem:cl}, \ref{it:cl2} and \ref{it:cl6},
that $Y\in \CS(\G)$ if and only if $Y=\cO^X(U)$, for some subset $U$ of
$X$. If $Y$ is closed and $Y=\cO^X(U)$ then $Y=\cl(Y)=\cO^X(\cl(U))$ and,
as $\cl(U)$ is closed, the result follows.
\item From the previous statement it follows that $\cO^X(X)\in \CS(\G)$.
If $Y\in \CS(\G)$ then $Y=\cO^X(U)$, for some $U\subseteq X$. From
Lemma \ref{lem:orth} then $\cO^X(X)\subseteq \cO^X(U)=Y$, as required.
\item From Lemma \ref{lem:cl}, \ref{it:cl1}
and \ref{it:cl5}, we have
\[Y_1\cap Y_2\subseteq \cl(Y_1\cap Y_2)\subseteq \cl(Y_1)\cap \cl(Y_2)
=Y_1\cap Y_2,\]
the last equality holding by definition of closed set. Therefore
$Y_1\cap Y_2=\cl(Y_1\cap Y_2)$.
\ee
\end{proof}
The
relation $Y_1\subseteq Y_2$ defines a partial order on the set
$\CS(\G)$. As the closure operator $\cl$ is inclusion preserving and maps
arbitrary subsets of $X$ into closed sets we can make
$\CS(\G)$ into a lattice by defining the
 the infimum  $Y_1\wedge Y_2$ of $Y_1$ and $Y_2$
by  $Y_1\wedge Y_2=\cl(Y_1\cap Y_2)=Y_1\cap Y_2$
and the supremum
$Y_1 \vee Y_2=\cl(Y_1\cup Y_2)$.

\begin{prop} \label{prop:cslattice}
The set $\CS(\G)$ with operations $\wedge$ and $\vee$ above is a
complete    lattice.
\end{prop}
\begin{proof}
As we have seen $\CS(\G)$ is a lattice. From Lemma \ref{lem:cs} it
has maximum element $X$ and minimum element $\cO^X(X)$, so is complete.
\end{proof}
\begin{prop}\label{prop:orthmap}
The  operator $\cO^X$ maps  $\CS(\G)$ to itself and is a lattice
anti-automorphism.
\end{prop}
\begin{proof}
If $Y\in \CS(\G)$ then, from Lemma \ref{lem:cs}, $\cO^X(Y)\in \CS(\G)$;
so $\cO^X$ maps $\CS(\G)$ to itself. From Lemma \ref{lem:orth} $\cO^X$
is inclusion reversing. Moreover, for $Y\in \CS(\G)$ we
have $\cO^X(\cO^X(Y))=Y$; so the restriction of $\cO^X$ to
$\CS(\G)$ is a bijection. Hence this restriction is a lattice
anti-automorphism.
\end{proof}

If $Z\subseteq X$ and $\G_Z$ is the full subgraph of $\G$ with vertex
set $Z$ then, by abuse of notation,  we write $\CS(Z)$ for
$\CS(\G_Z)$. As long as it is clear that $\G$ is fixed this should cause
no confusion. We have $\cO^Z(Y)=\cO^X(Y)\cap Z$ so $\CS(Z)$ consists
of subsets $Y$ of $Z$ such that $Y=\cl^Z(Y)=\cO^X(\cO^X(Y)\cap Z)\cap Z$. 

\subsection{Disconnected Graphs and Joins of Graphs}\label{sec:DG}
Now suppose that $X$ is a disjoint union $X=X_1\sqcup X_2$,
where $X_1$ and $X_2$ are non-empty, and $\G=\G(X_1)\sqcup \G(X_2)$
(that is no edge of $\G$ joins a vertex of $X_1$ to a vertex of $X_2$).
Write $\G_i$ for $\G(X_i)$, $i=1,2$. We wish to describe
$\CS(\G)$ in terms of the lattices $\CS(\G_i)$. First of all we note
the following lemma.
\begin{lem}\label{lem:disjoint} With the hypotheses above,
if $U$ is a non-empty subset of $X_i$ then $\cO^{X_i}(U)=\cO^{X}(U)$.
\end{lem}
\begin{proof}
By definition  $\cO^{X_i}(U)\subseteq \cO^{X}(U)$.
We have $\cO^{X}(U)=\{x\in X|d(u,x)\le 1, \forall u\in U\}$.
If $x\notin X_i$ then, as $U\neq \nul$, there is some $u\in U$ such
that $d(x,u)=\infty$. Hence $x\in \cO^{X}(U)$ implies $x\in X_i$, so
$x\in \cO^{X_i}(U)$.
\end{proof}
The relationship between $\CS(\G)$ and the $\CS(\G_i)$ is
specified by the following proposition.
\begin{prop}\label{prop:disjoint}
Let $\G=\G_1\sqcup \G_2$, as above.
\be
\item $\nul \in \CS(\G)$.
\item A non-empty set $Y$ is in $\CS(\G)\backslash \{X,X_1,X_2\}$ if and only if
$Y$ is in precisely one of $\CS(\G_i)\backslash \{X_i\}$,
$i=1$ or $2$.
\item If $\cO^{X_i}(X_i) =\nul$ then $\nul \in CS(\G_i)$ and
$X_i\notin CS(\G)$.
\item If $\cO^{X_i}(X_i) \neq \nul$ then $\nul \notin CS(\G_i)$ and
$X_i\in CS(\G)$.
\ee
\end{prop}
\begin{proof}
\be
\item As $X_i$ is non-empty it follows that $\nul=\cO^X(X)$, so
$\nul\in \CS(\G)$.
\item Let $Y$ be a non-empty element of $\CS(\G)\backslash\{X, X_1,X_2\}$.
Then $Y=\cO^X(U)$, for some subset $U$ of $X$. If $U\cap X_i\neq \nul$,
for $i=1$ and $2$, then $\cO^X(U)=\nul$. Hence $U\subseteq X_i$, for $i=1$ or $2$.
If $U=\nul$ then $Y=X$, so $U\neq \nul$ and, from Lemma \ref{lem:disjoint},
$Y=\cO^{X_i}(U)$ so is in $\CS(\G_i)$. Note that in this case
$Y\subseteq X_i$ and is  non-empty; so cannot be in $\CS(\G_j)$, $j\neq i$.
Conversely if $Y$ is a non-empty element of $\CS(\G_i)\backslash\{X_i\}$
then $Y=\cO^{X_i}(U)$, for some $U\subseteq X_i$. As $Y\neq X_i$ we have
$U\neq \nul$ and so, from Lemma \ref{lem:disjoint} again,
$Y\in CS(\G)$.
\item From Lemma \ref{lem:cs}, $\nul \in \CS(\G_i)$. From Lemma
\ref{lem:disjoint} we have $\nul=\cO^{X_i}(X_i)=\cO^X(X_i)$.
If $X_i\in \CS(\G)$ then $X_i=\cO^X(U)$, for some $U\in \CS(\G)$.
Hence $\nul=\cO^X(X_i)=U$ which implies $X_i=\cO^X(U)=X$, a contradiction.
\item As $\cO^{X_i}(X_i)$ is the minimal element of $\CS(\G_i)$, in this
case $\nul\notin \CS(\G_i)$. We have $X_i=\cO^{X_i}(U)$, for
some $U\in \CS(\G_i)$, so $U\neq \nul$ and $U\subseteq X_i$. That
$X_i\in \CS(\G)$ now follows from Lemma \ref{lem:disjoint}.
\ee
\end{proof}

Let $L=\CS(\G)$, $L_i=\CS(\G_i)$ and $L_i^\prime=\CS(\G_i)\backslash\{X_i\}$.
Then Figure \ref{fig:lattice2} illustrates
the composition of $\CS(\G)$ in terms of the
$\CS(\G_i)$.
\begin{figure}[!h]
  \centering
  \psfrag{X}{$X$}
  \psfrag{L1}{$L_1$}
  \psfrag{L'2}{$L_2^\prime$}
  \psfrag{X1}{$X_1$}
  \psfrag{X1t}{$X_1^\perp$}
  \psfrag{L2}{$L_2$}
  \psfrag{L'1}{$L_1^\prime$}
  \psfrag{X2}{$X_2$}
  \psfrag{X2t}{$X_2^\perp$}
  \psfrag{N}{$\emptyset$}
\parbox{0.3\textwidth}
{
\begin{center}
   \includegraphics[scale=0.3]{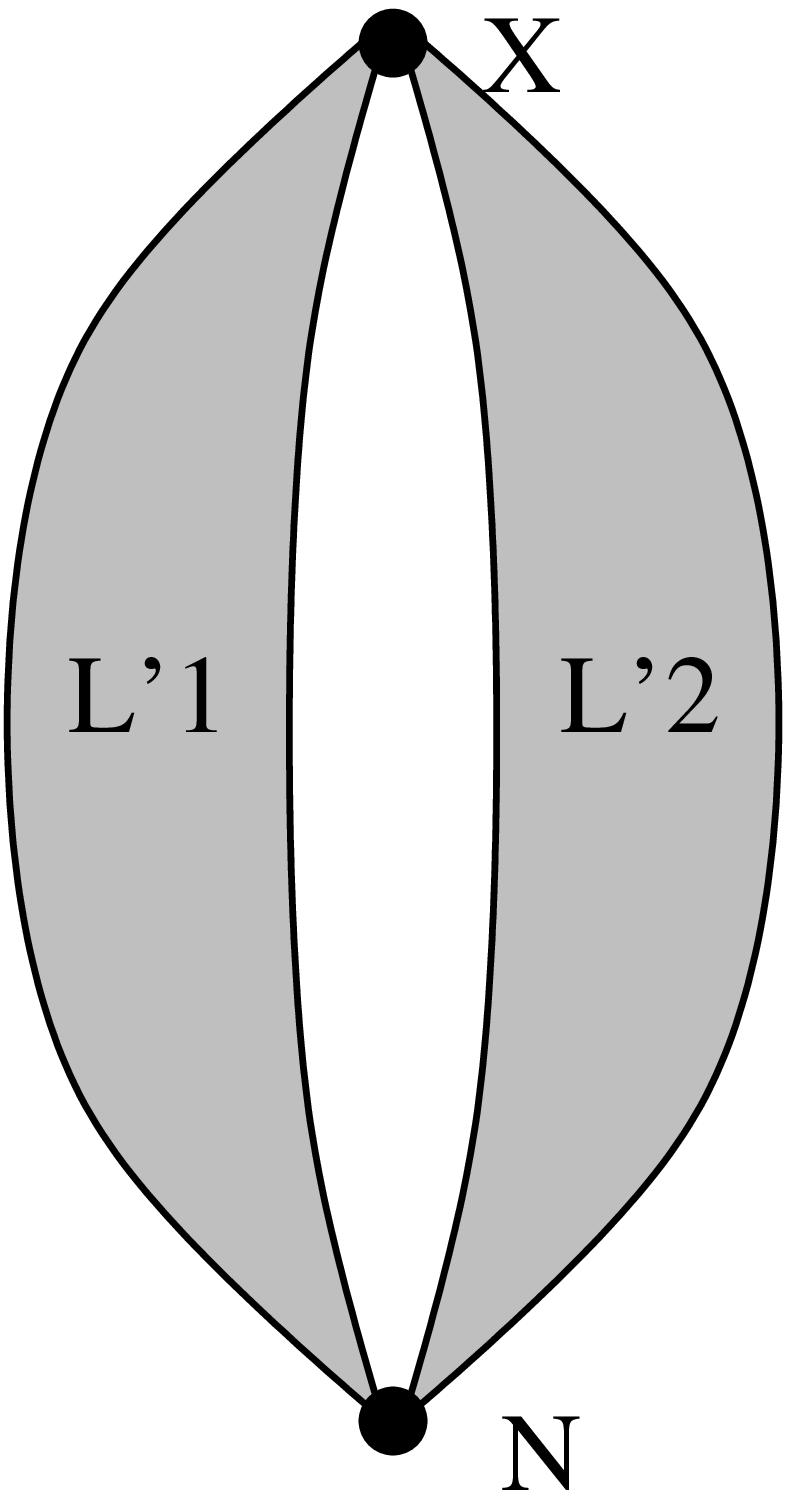}\\[.5em]
$\cO^{X_i}(X_i)=\nul$,\\ $i=1,2$
\end{center}
}
\parbox{0.3\textwidth}
{
\begin{center}
   \includegraphics[scale=0.3]{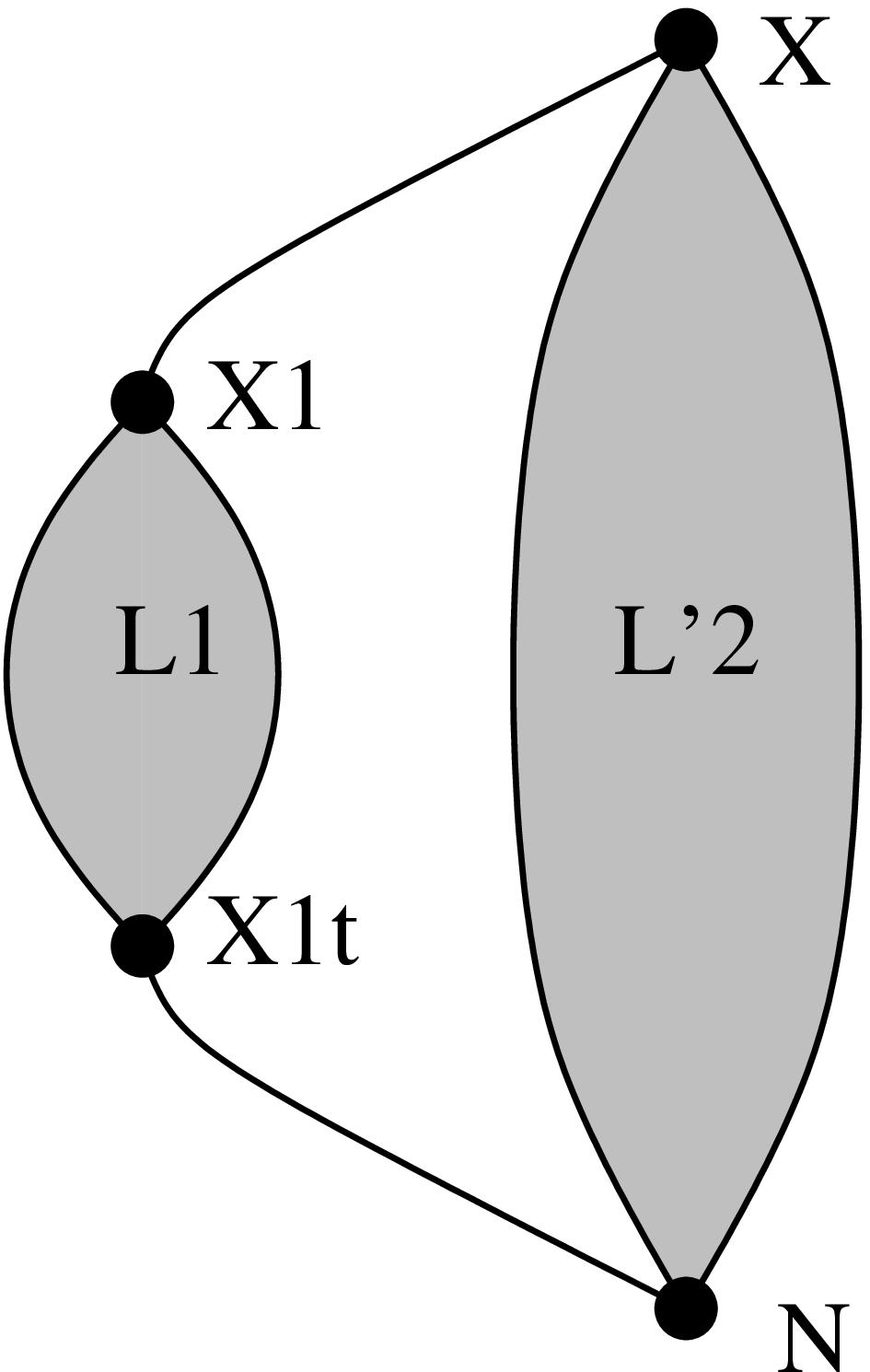}\\[.5em]
$\cO^{X_1}(X_1)\neq\nul$,\\
$\cO^{X_2}(X_2)=\nul$,
\end{center}
}
\parbox{0.3\textwidth}
{
\begin{center}
   \includegraphics[scale=0.3]{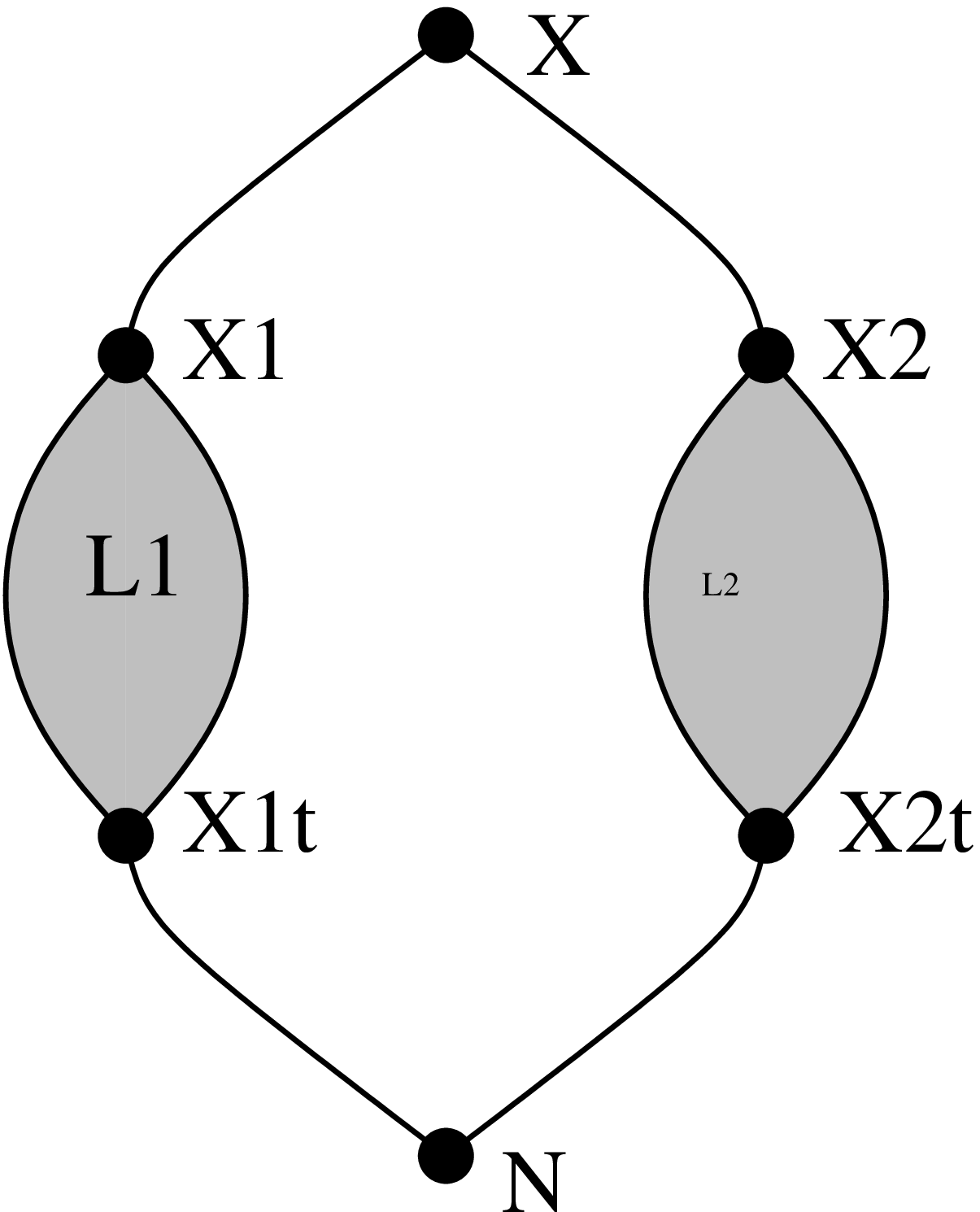}\\[.5em]
$\cO^{X_i}(X_i)\neq\nul$,\\ $i=1,2$
\end{center}
}
 \caption{The lattice $L$ of closed sets in  a disconnected graph} \label{fig:lattice2}
 \end{figure}
Now suppose that
$\G$  has connected components
$\G_1,\ldots ,\G_m$, where $V(\G_i)=X_i$. Assume that
$\cO^{X_i}(X_i)\neq \nul$, for $i=1,\ldots ,r$ and that
$\cO^{X_i}(X_i)= \nul$, for $i>r$.
A straightforward induction using Proposition \ref{prop:disjoint}
shows that
the lattice $\CS(\G)$ takes
the form shown in Figure \ref{fig:lattice3}: where we use the obvious
extension of the notation introduced above for the lattices $\CS(\G_i)$.
  \begin{figure}[!h]
    \centering
    \psfrag{X}{$X$}
    \psfrag{L1}{$L_1$}
    \psfrag{Lr}{$L_r$}
    \psfrag{L'r1}{$L_{r+1}^\prime$}
    \psfrag{X1}{$X_1$}
    \psfrag{X1t}{$X_1^\perp$}
    \psfrag{Lr}{$L_r$}
    \psfrag{L'1}{$L_1^\prime$}
    \psfrag{L'm}{$L_m^\prime$}
    \psfrag{Xr}{$X_r$}
    \psfrag{Xrt}{$X_r^\perp$}
    \psfrag{N}{$\emptyset$}
    \begin{center}
      \includegraphics[scale=0.4]{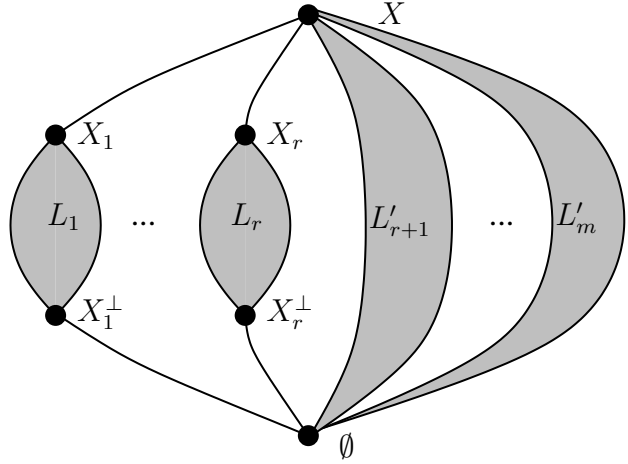}
    \end{center}
    \caption{The lattice $L$ of  the graph with
      connected components $\G_1,\ldots ,\G_m$.}
    \label{fig:lattice3}
  \end{figure}
We may often therefore reduce to the study of $\CS(\Gamma)$ where $\G$ is
a connected graph.

Now suppose that $X^\perp \ne \emptyset$ and set $X^*=X\backslash
X^\perp$. Let $\G(X^*)=\Gamma^*$ the full subgraph of $\Gamma$ with
vertex set $X^*$.
\begin{prop}
The set $\cO^{X^*}(X^*)=\nul$ and
the  lattice $\CS(\G)$ is isomorphic to the lattice $\CS(\G^*)$.
\end{prop}
\begin{proof} From the definitions it follows that $\cO^X(X^*)=\cO^X(X)$.
Therefore $\cO^{X^*}(X^*)=\cO^{X}(X^*)\cap X^*= \cO^{X}(X)\cap X^*=\nul$. 
 If $Y=\cO^X(U)$, where
$U\in \CS(\G)$ then
$Y\backslash X^\perp =\cO^X(U)\backslash \cO^X(X)
=\cO^{X^*}(U\backslash \cO^X(X)).
$
Hence
the map $\phi:Y \to Y
\backslash X^\perp$
maps $\CS(\G)$ to $\CS(\G^*)$.

Clearly $\phi$ is inclusion preserving. To see that $\phi$ is surjective, 
note that if $V\subseteq X^*$ then $\cO^{X^*}(V)=C\backslash X^\perp$, where
$C=\cO^X(V\backslash \cO^X(X))$. Therefore
$\phi$ is a surjective homomorphism of partially  ordered sets.
Since $Y\in \CS(\G)$ implies $X^\perp\subseteq Y$ it follows
that $\phi$ is also injective; so $\phi$ is an isomorphism of lattices.
\end{proof}
The set $\cO^X(X)$ is called the \textit{kernel} of the graph
$\Gamma$. Given the proposition above we may restrict to
the study of lattices
with the trivial kernel.

Now suppose that $\G=\G(X_1)\oplus \G(X_2)$, for some partition
$X=X_1\cup X_2$ of $X$ (see Section \ref{sec:graph}).  
Let $\G_i=\G(X_i)$ and let $G_i=G(\G_i)$,
$i=1,2$; so
$G=G_1\times G_2$. 
\begin{prop}
In the above notation, if $\G=\G_1\oplus \G_2$ then 
$\CS(\G)=\CS(\G_1)\times\CS(\G_2)$.
\end{prop}
In this case  the study of the lattice $\CS(\G)$ reduces to the study of
$\CS(\G_1)$ and $\CS(\G_2)$.
\subsection{Adjoining Vertices}\label{sec:av}

We now consider the effect on the lattice of closed sets of the
addition to $\G$, or removal from $\G$, of a vertex. In particular we
shall see how the heights of these lattices are related and
how
to make restrictions on the way in which the new vertex is added to
obtain isomorphism of the two lattices.

We shall see below that if we adjoin a single vertex to $\G$ then the height
of the lattice of closed sets of the new graph is equal to
$h(\CS(\G))+k$, where $k=0,1$ or $2$.

As usual $\G$ is a graph with $V(\G)=X$ and edges $E(\G)$.
Let $t$ be an element not in $X$
and define $\ov{X}=X\cup\{t\}$. Let $J_t$ be a subset of $X$.
Define $\ov\G$ to be the graph with vertices
$\ov X$ and edges $E(\G)\cup E_t$, where $E_t$ is the set
$E_t=\{(t,x)|x\in J_t\}$.
Let $L=\CS(\G)$ and $\ov L=\CS(\ov \G)$.

In order to understand how $L$ and $\ov L$ are related
we introduce a lattice intermediate between $L$ and $\ov L$.
This will help us to give a simple description of the structure
of the lattice $\ov L$ in terms of the lattice $L$.
Let
\[L_t=\{Y\subseteq X| Y=C\cap J_t, \textrm{ where } C\in L\}.\]
Now define the set of subsets $\tilde L$ of $X$ to be
$\tilde L=L\cup L_t$. We shall see that $\tilde L$ is a lattice
which embeds in the lattice $\ov L$. 
Note that if $Y\in L_t$ then $Y=C\cap J_t$, for some $C\in L$, so
\[Y=Y\cap J_t\subseteq \cl(Y)\cap J_t\subseteq \cl(C)\cap J_t=C\cap J_t=Y.\]
Hence $Y=\cl(Y)\cap J_t$ and it follows that $\cl(Y)$ is the minimal
element of $L$ which intersects with $J_t$ to give $Y$, for all $Y\in L_t$. 
Setting $Z=\cl(Y)$ this gives
$Z=\cl(Y)=
\cl(\cl(Y)\cap J_t)=\cl(Z\cap J_t)$.
Also if $Y\in L_t\backslash L$ then $\cl(Y)\neq Y=\cl(Y)\cap J_t$, so $Z=\cl(Y)\nsubseteq J_t$.
Conversely, given $Z\in L$ such that $Z\nsubseteq J_t$ and $Z=\cl(Z\cap J_t)$ then
$\cl(Z\cap J_t)\neq Z\cap J_t$, so $Z\cap J_t\in L_t\backslash L$. Therefore
\begin{equation}\label{eq:L_t}
L_t\backslash L=\{Y=Z\cap J_t|Z\in L, Z\nsubseteq J_t \textrm{ and } Z=\cl(Z\cap J_t)\}.
\end{equation}
We define a closure operation $\icl^X=
\icl$ on subsets of $X$ by
\[
\icl(U)=
\left\{
\begin{array}{ll}
\cl^X(U), & \textrm { if } U\nsubseteq J_t\\
\cl^X(U)\cap J_t, & \textrm { if } U\subseteq J_t
\end{array}
\right.
,
\]
for $U\subseteq X$. Then $\icl(U)\in \tilde L$, for all $U\subseteq X$.

Now assume that  $Y_1$ and $Y_2$ are in $\tilde L$ and
$Y_1\subseteq Y_2$. If $Y_1\nsubseteq J_t$ then $\icl(Y_1)=\cl^X(Y_1)$ and
$\icl(Y_2)=\cl^X(Y_2)$ so $\icl(Y_1)\subseteq \icl(Y_2)$. If $Y_1\subseteq J_t$
then $\icl(Y_1)=\cl^X
(Y_1)\cap J_t\subseteq \cl^X
(Y_2)\cap J_t\subseteq \icl(Y_2)$.
Therefore $\icl$ is an inclusion preserving map from subsets of $X$ to 
$\tilde L$. It also follows from the definition and \ref{eq:L_t} that 
$\icl(U)=U$, for all $U\in \tilde L$, so $\tilde L$ is a retract of $X$.
We may therefore make $\tilde L$ into a lattice by defining
\[Y_1{\wedge} Y_2=\icl(Y_1\cap Y_2)\textrm{ and } Y_1{\vee} Y_2=\icl(Y_1\cup Y_2),\]
for $Y_1,Y_2\in \tilde L$.
\begin{lem}\label{lem:iclinf}
If $U, V\in \tilde L$ then $U\wedge V=U\cap V$
and
\[
U\vee V=
\left\{
\begin{array}{ll}
\cl(U\cup V)\cap J_t, & \textrm{ if } U\cup V\subseteq J_t\\
\cl(U\cup V), &\textrm{ otherwise}
\end{array}
\right. .
\]
\end{lem}
\begin{proof}
The expression for $U\vee V$ is merely a restatement of the definitions.
If   $U\in L$ then $\icl(U)=\cl^X
(U)$. Therefore, for $U$ and $V$ in
$L$ we have (in the lattice $\tilde L$)
$U{\wedge} V=U\cap V$. If either $U$ or $V$ belongs to $L_t$ then
$U\cap V\subseteq J_t$ so
\[U{\wedge} V=\cl(U\cap V)\cap J_t\subseteq \cl(U)\cap \cl(V)\cap J_t=U\cap V\subseteq
\cl(U\cap V)\cap J_t\]
and  the Lemma follows.
\end{proof}

\begin{defn}\label{defn:ebg}
Define $\tilde \b$ to be the inclusion map of $L$ into $\tilde L$ and
$\tilde \g$  to be the map from $\tilde L$ to $L$ given by
$\tilde \g(Y)=\cl^X(Y)$, for $Y\in \tilde L$.
\end{defn}
\begin{lem}\label{lem:ebg}
The maps $\tilde \b$ and $\tilde \g$ are homomorphisms of partially ordered
sets and  $\tilde \g\tilde \b=\id_L$. We have $\tilde \b(Y\wedge Z)=\tilde \b(Y)\wedge \tilde \b(Z)$,
for all $Y,Z\in L$, and
$\tilde \g(U\vee V) =\tilde \g(U)\vee \tilde \g(V)$, for all $U,V\in \tilde L$.
If $U,V\in \tilde L$ such that $U\neq V$ and $\tilde \g(U)=\tilde \g(V)$ then
(after interchanging $U$ and $V$ if necessary) $U\in L\backslash L_t$ and
$V\in L_t\backslash L$ and $U=\cl^X(V)$.
\end{lem}
\begin{proof}
The first statement is a direct consequence of the definitions, as is the
fact that $\tilde \b$ respects the lattice infimum operation.
For all $U,V\in \tilde L$ we have
\[\tilde \g(U)\vee \tilde \g(V)=\cl(\cl(U)\cup\cl(V))=\cl(U\cup V),\]
from Lemma \ref{lem:cl}.\ref{it:cl10}.
If $U\cup V\nsubseteq J_t$ then
$\tilde \g(U\vee V)=\cl\cl(U\cup V)=\cl(U\cup V)$. On the other hand,
if $U\cup V\subseteq J_t$ then
$\tilde \g(U\vee V)=\cl(\cl(U\cup V)\cap J_t)=\cl(U\cup V)$, using
Lemma \ref{lem:cl}.\ref{it:cl9}. Hence $\tilde \g(U)\vee \tilde \g(V)=
\tilde \g(U\vee V)$,
 for all $U,V\in \tilde L$.

Let $U,V\in \tilde L$. If $U,V\in L$ then $\tilde \g(U)=\tilde \g(V)$ implies
$U=V$. If $U,V\in L_t$ then $U=\cl(U)\cap J_t$ and $V=\cl(V)\cap J_t$ and
$\cl(U)=\cl(\cl(U)\cap J_t)=\tilde \g(U)$ and similarly $\cl(V)=\tilde \g(V)$.
Therefore $\tilde \g(U)=\tilde \g(V)$ implies that
$U=\cl(U)\cap J_t=\cl(V)\cap J_t=V$. Therefore, if $U\neq V$ and
$\tilde \g(U)=\tilde \g(V)$ then one of $U,V$ is in $L\backslash L_t$ and
the other in
$L_t\backslash L$. Assume then that $U\in L\backslash L_t$ and
$V\in L_t\backslash L$. In this case $U=\tilde \g(U)=\tilde \g(V)
=\cl(\cl(V)\cap J_t)=\cl(V)$.
\end{proof}
  In general $\tilde \b$ does not preserve supremums and
$\tilde \g$ does not preserve infimums.
\begin{expl}\label{ex:path7}
In the graph of  Figure \ref{fig:path7} the sets $B=\{b\}$ and $C=\{c\}$ are
closed.  The supremum $B\vee C=\cl(B\cup C)=\{b,c,y\}$ and setting
$J_t=\{b,c\}$ we have
$\tilde \b(B\vee C)=\{b,c,y\}$ and $\tilde\b(B)\vee \tilde\b(C)=
\cl(\{b,c\})\cap \{b,c\}=\{b,c\}$. In the same graph
$\cl(x)=\{a,x,c\}$ and $\cl(y)=\{b,y,c\}$.
Set
$J_t=\{x,y\}$ and then $U=\cl(x)\cap J_t=\{x\}$ and $V=\cl(y)\cap J_t
=\{y\}$ are both elements of $L_t$. Now $U\wedge V=\emptyset$ so
$\tilde \g(U\wedge V)=\cl(\emptyset)=\emptyset$. However
$\tilde\g(U)\wedge \tilde \g(V)=\cl(x)\cap \cl(y)=\{c\}$.
\end{expl}
\begin{figure}
\psfrag{a}{$a$}
\psfrag{b}{$b$}
\psfrag{c}{$c$}
\psfrag{e}{$e$}
\psfrag{f}{$f$}
\psfrag{x}{$x$}
\psfrag{y}{$y$}
\begin{center}
\includegraphics[scale=0.4]{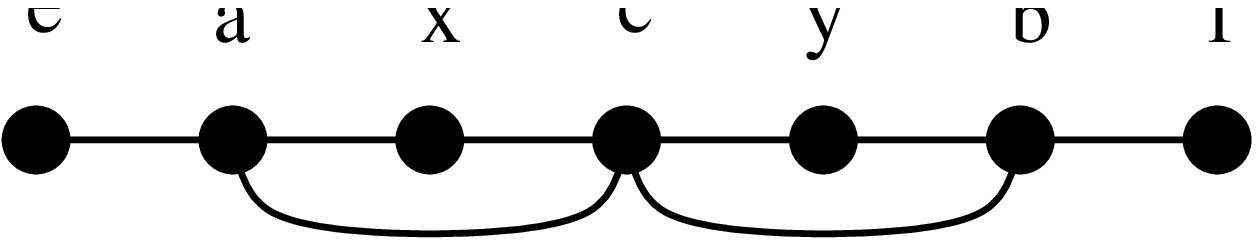}
\end{center}
\caption{Example \ref{ex:path7}}\label{fig:path7}
\end{figure}

Next we show that the lattice $\tilde L$ is embedded, as a partially ordered
set,
in $\ov L$.
\begin{defn}\label{defn:ovbg}
Let $\ov \b:\tilde L\maps \ov L$ and $\ov \g:\ov L\maps \tilde L$
be the maps  given by $\ov \b(Y)=\cl^{\ov X}(Y)$, for $Y\in \tilde L$
and $\ov \g(Z)=\icl(Z\backslash \{t\})$,
for $Z\in \ov L$.
\end{defn}
\begin{lem}\label{lem:ovbg} The maps $\ov \b$ and $\ov \g$ are homomorphisms
of partially ordered sets and $\ov \g\ov \b= \id_{\tilde L}$; so
$\ov \b$ is injective and $\ov \g$ is surjective. We have
$\ov \g(Z)=Z\backslash \{t\}$, for all $Z\in \ov L$, and
 \begin{equation}\label{eq:ovb}
\ov \b(Y)=
\left\{
\begin{array}{ll}
Y & \textrm{ if } \cO^X(Y)\nsubseteq J_t,\\
Y\cup \{t\} & \textrm{ if } \cO^X(Y)\subseteq J_t
\end{array}
\right.
,
\end{equation}
for all $Y\in \tilde L$.

If $Z_1$ and $Z_2$ are elements of $\ov L$ such that $Z_1\neq Z_2$ then
$\ov \g(Z_1)=\ov \g(Z_2)$ if and only if  (after interchanging $Z_1$ and $Z_2$
if necessary)  
$t\in Z_1$ and $Z_2=Z_1\backslash \{t\}\in \ov L$.
\end{lem}
\begin{proof}
Since the closure operations in $\tilde L$ and $\ov L$ preserve inclusion
of sets it follows from the definitions that
$\ov \b$ and $\ov \g$ are homomorphisms
of partially ordered sets.

Now let $U\in \tilde L$. If $U\nsubseteq J_t$
then $\cO^{\ov X}(U)=\cO^X(U)$.  On the other hand if $U\subseteq J_t$
then $\cO^{\ov X}(U)=\cO^X(U)\cup \{t\}$.
Therefore, if  $U\nsubseteq J_t$ then
\[
\ov \b(U)=\cO^{\ov X}\cO^{X}(U)=
\left\{
\begin{array}{ll}
\cO^{X}\cO^{X}(U), & \textrm{ if } \cO^{X}(U)\nsubseteq J_t\\
\cO^{X}\cO^{X}(U)\cup \{t\}, & \textrm{ if } \cO^{X}(U)\subseteq J_t
\end{array}
\right.
\]
and \eqref{eq:ovb} holds as $U\nsubseteq J_t$ implies that $U\in L$.
If $U\subseteq J_t$ then
\[\ov \b(U)=\cO^{\ov X}(\cO^{X}(U)\cup \{t\})=
\cO^{\ov X}(\cO^{X}(U))\cap(J_t\cup \{t\})\]
so
\[
\ov \b(U)=
\left\{
\begin{array}{ll}
\cO^{X}\cO^{X}(U)\cap J_t, & \textrm{ if } \cO^{X}(U)\nsubseteq J_t\\
(\cO^{X}\cO^{X}(U)\cap J_t)\cup \{t\}, & \textrm{ if } \cO^{X}(U)\subseteq J_t
\end{array}
\right.
.
\]
In this case, as $U\subseteq J_t$ we have $\cO^{X}\cO^{X}(U)\cap J_t=
\cl^X(U)\cap J_t=U$. Thus, in all cases, \eqref{eq:ovb}
holds.

Now suppose that $Z\in \ov L$ and let $Y\in \ov L$ such that
$Z=\cO^{\ov X}(Y)$. If $t\in Y$ then
$Z=\cO^{\ov X}(Y)\subseteq  \cO^{\ov X}(t)=J_t\cup \{t\}$.
 Conversely if $Z\subseteq J_t\cup \{t\}$ then
$t\in Y=\cO^{\ov X}(Z)$. Hence
$Z\backslash \{t\} \subseteq J_t$ if and only if $t\in Y$.
Similarly $Y\backslash \{t\} \subseteq J_t$ if and only if $t\in Z$.
To show that $\ov \g(Z)=Z\backslash \{t\}$ we consider various cases.
\be
\item
Suppose that $t\in Z$ and that $t\notin Y$. Then
$Y\subseteq J_t$ and $Z=\cO^{\ov X}(Y)=\cO^X(Y)\cup \{t\}$.  Therefore
$Z\backslash \{t\}= \cO^X(Y)\in L$ and, since
$Z\backslash \{t\}\nsubseteq J_t$, it follows that
$\ov \g(Z)=\icl(Z\backslash \{t\})=Z\backslash \{t\}$.
\item
Assume that $t\in Z$ and $t\in Y$.
 Then
$Y\subseteq J_t\cup\{t\}$ and
\begin{align*}
Z
&=\cO^{\ov X}(Y)=\cO^{\ov X}((Y\backslash\{t\})\cup \{t\})\\
&=\cO^{\ov X}(Y\backslash\{t\})\cap  \cO^{\ov X}(t)\\
&=(\cO^X(Y\backslash\{t\})\cup \{t\})\cap (J_t \cup \{t\})\\
&=(\cO^X(Y\backslash\{t\})\cap J_t)\cup \{t\}.
\end{align*}
  Therefore
$Z\backslash \{t\}= \cO^X(Y\backslash\{t\})\cap J_t$ and, since
$Z\backslash \{t\}\subseteq J_t$, we have,
using Lemma \ref{lem:cl}.\ref{it:cl11},
$\icl(Z\backslash \{t\})=\cl(\cO^X(Y\backslash\{t\})\cap J_t)\cap J_t =
\cO^X(Y\backslash\{t\})\cap J_t$. Therefore
$
\ov \g(Z)=\icl(Z\backslash \{t\})=
Z\backslash \{t\}$.
\item Assume that $t\notin Z$ and $t\notin Y$. In this case
$Z= \cO^X(Y)\in L$ and, since
$Z\nsubseteq J_t$, it follows that
$\ov \g(Z)=\cl(Z)=Z=Z\backslash \{t\}$.
\item Assume that $t\notin Z$ and $t\in Y$. Since $t\notin Z$ this
means that $Z\subseteq J_t$ and
$\ov \g(Z)=\cl(Z)\cap J_t$. Now 
\begin{align*}
Z
&=\cO^{\ov X}(Y)=\cO^{\ov X}(Y\backslash\{t\})\cap  (J_t\cup \{t\})\\
&=\cO^X(Y\backslash\{t\})\cap J_t\in L_t,
\end{align*}
as $Y\backslash \{t\}\nsubseteq J_t$. Hence $Z=\cl(Z)\cap J_t$
and so $\ov \g(Z)=Z
=Z\backslash \{t\}$.
\ee
Thus $\ov \g(Z)=Z\backslash \{t\}$, for all $Z\in \ov L$.

Now suppose that $Z_1,Z_2\in \ov L$ such that $Z_1\neq Z_2$. Suppose that
$\ov \g(Z_1)
=\ov \g(Z_2)$. As $\g(Z_i)=Z_i\backslash \{t\}$ we must have, after
interchanging $Z_1$ and $Z_2$ if necessary, $Z_1=Z_2\cup \{t\}$;
so $t\in Z_1\in \ov L$ and $Z_1\backslash\{t\}\in \ov L$.
\end{proof}
\begin{defn}
Let $\b:L\maps \ov L$ be the map given by $\b(Y)=\cl^{\ov X}(Y)$, for
$Y\in L$. Let $\g:\ov L\maps L$ be the map given by $\g(Z)=\cl^X(X\cap Z)$,
for $Z\in \ov L$.
\end{defn}
\begin{cor}\label{cor:gb}
We have $\b=\ov \b\tilde \b$ and $\g=\tilde \g\ov \g$.
The maps $\b$ and $\g$ are homomorphisms of partially ordered sets.
For $Y\in L$ 
\[
\b(Y)=
\left\{
\begin{array}{ll}
Y & \textrm{ if } \cO^X(Y) \nsubseteq J_t,\\
Y\cup \{t\} & \textrm{ if } \cO^X(Y)\subseteq J_t
\end{array}
\right.
.
\]
Moreover $\g\b=\id_L$, $\b$ is injective and $\g$ is surjective.
\end{cor}
\subsection{The height of the extended lattice}\label{sec:HE}

In this section we determine  the possible differences
in height between the lattices $L$ and $\ov L$.
By a {\em strong ascending chain} in a partially ordered set $L$ is meant a
sequence $C_0, C_1\ldots $ of elements of $L$ such that
$C_i< C_{i+1}$, for all $i\ge 0$. {\em Strong descending chains} are defined
analogously, replacing $<$ by $>$. The {\em length} of a finite strong chain
$C_0, \ldots , C_d$ is $d$.
If $C_0, C_1\ldots $ is a sequence of elements of $L$ such that
$C_i\le C_{i+1}$, for all $i\ge 0$, then we call $C_0, C_1\ldots $
a {\em weak} ascending chain. Weak descending chains are defined analogously.
The {\em length} of a weak chain $\cC$ is the maximum of the lengths of
strong chains obtained by taking subsequences of $\cC$.
We shall from now on use
chain to mean either weak or strong chain, if the meaning is clear.
We denote the length of a chain $\cC$ by $l(\cC)$.
Let $L$ and $L^\prime$ be partially ordered sets and let $\phi:L\maps L^\prime$
be a homomorphism or anti-homomorphism of partially ordered sets.
If $\cC$ is a chain $C_0, \ldots , C_d$ in $L$ then
we denote by $\phi(\cC)$ the chain $\phi(C_0), \ldots ,\phi(C_d)$, in
$L^\prime$. Clearly the length of $\cC$ is greater than or equal to
the length of $\phi(\cC)$.
\begin{defn}
The {\em height} $h(L)$ of a lattice $L$ is
the
length of its maximal chain, if this exists, and is infinite otherwise.
\end{defn}

The following is a corollary of Lemmas \ref{lem:ebg} and \ref{lem:ovbg}
\begin{cor}\label{cor:hL}
$h(L)\le  h(\tilde L)\le h(\ov L)$.
\end{cor}
\begin{proof}
If $\cC$ is a maximal chain in $L$ then $\tilde \b(\cC)$ is a chain in $\tilde L$.
As $\tilde \b$ is injective $\tilde \b(\cC)$ has the same length as $\cC$ and the
result follows. The second inequality follows similarly.
\end{proof}
\begin{expl}
Let $\G$ be the graph of Figure \ref{fig:path3} and let $J_t=\{a,c\}$.
Then $L$ consists of  $X$, the
orthogonal complements (in $X$) of $a$, $b$, $c$ and $d$,  and also
$\{b,c\}=\cO^{X}\{b,c\}$, $\{b\}=\cO^{X}\{a,c\}$,
$\{c\}=\cO^{X}\{b,d\}$ and $\emptyset$. Therefore $h(L)=4$. $\tilde L$ 
contains in addition the set $J_t$ and the set $\{a\}=J_t\cap \cO^X(a)$.
It follows that $h(\tilde L)=4$ as well. Finally, the maximal proper
subsets of $\ov L$ are
the orthogonal complements (in $\ov X$) 
of $a$, $b$, $c$ and $t$ (as $\cO^{\ov X}(d)\subseteq \cO^{\ov X}(c)$).
The only one of these sets with $4$ elements is $\cO^{\ov X}(c)$. 
However, the intersection of $\cO^{\ov X}(c)$ with any other proper
maximal subset has at most $2$ elements. Hence $\ov L$ can have
height at most $4$. As $h(\tilde L)=4$ it now follows that 
$h(\ov L)=h(\tilde L)=h(L)=4$.
\end{expl}

\begin{expl}
Let $\ov \G$ be the graph of Figure \ref{fig:path3} and $\G$ be
the graph obtained by removing vertex $c$. Then, with $t=c$ we 
have $X=\{a,b,d\}$ and  $J_t=\{b,d\}$. In this case $L$ consists
of the sets $X$, $\cO^X(a)$, $\cO^X(d)$ and $\emptyset$, so $h(L)=2$. 
$L_t$ contains in addition the sets $J_t$ and $\cO^X(a)\cap J_t=\{b\}$.
Thus $h(\tilde L)=3$. Moreover, from the previous example 
 $h(\ov L)=4$.
(The {\em semibraid} group on $n$ generators is the partially 
commutative group $G_n$ with presentation 
\[\la x_1,\ldots ,x_n| [x_i,x_j]=1, \textrm{ if } |i-j|\ge 2\ra.
\]
The graphs of this example are those of  $G_3$ 
and $G_4$, see \cite{DKR2} for further details)
\end{expl}
In fact these two examples illustrate the two extremes in 
differences of 
height between $L$ and $\tilde L$ and between $\tilde L$ and
$\ov L$: as the following propositions show.
\begin{prop}\label{prop:hintL}
$h(\tilde L)=h(L)+m$, where $m=0$ or $1$.
\end{prop}
\begin{proof}
Let
\[
\cC=Z_0<\cdots <Z_k
\]
be a strictly ascending chain in $\tilde L$, with $k=h(\tilde L)$.
Then $\tilde \g(\cC)$ is an ascending chain in $L$.
If $Z_i\in L$ for all $i$ then $\tilde \g(\cC)=\cC$,
so Lemma \ref{cor:hL}
implies that
$h(\tilde L)= h(L)$. Assume then that $Z_i\notin L$, for some $i$, and let
$r$ be the smallest integer such that $Z_r\in L$, for all $i\ge r$.
Then $Z_i\subseteq J_t$, 
so
$Z_i\in L_t$, for all $i\le r-1$.
 Using Lemma \ref{lem:ebg}, $\tilde\g(Z_r)<\cdots <\tilde\g(Z_k)$
and $\tilde \g(Z_0)<\cdots <\tilde \g(Z_{r-1})$ are strictly ascending chains in $L$.
The length of $\tilde \g(\cC)$ is therefore at least $k-1=h(\tilde L)-1$;
so $h(L)\ge h(\tilde L)-1$, and the lemma follows
from Lemma \ref{cor:hL}.
\end{proof}

\begin{prop}\label{prop:hovL}
$h(\ov L)=h(\tilde L) +m$, where $m=0$ or $1$.
\end{prop}
\begin{proof}
Let $\ov \cC=Z_0<\cdots <Z_k$ be a strictly ascending chain in $\ov L$. As
$\ov \g$ is inclusion preserving the sequence $\ov \g(\ov \cC)$ is ascending.
Let $r$ be the least integer such that $t\in Z_i$ for $r\ge i$. Then, 
from Lemma \ref{lem:ovbg},
\[
\ov \g(Z_0)<\cdots <\ov \g(Z_{r-1})\le \ov \g(Z_r)<\ov \g(Z_{r+1})<\cdots
<\ov \g(Z_k),
\]
so $\ov \g(\ov \cC)$ has length at least $k-1$.
\end{proof}

\begin{thm}\label{thm:hL}
$h(\ov L)=h(L)+m$, where $m=0,1$ or $2$.
\end{thm}
The next two examples show that a difference of one between the heights
of  $L$ and 
$\ov L$ may occur and may 
be due either to a difference in height between $L$ and 
$\tilde L$ or between $\tilde L$ and $\ov L$.
\begin{expl}\label{ex:lheight1}
Let $\G$ be the graph obtained by removing vertex $t$ from the
graph $\ov\G=\G_1$ of Figure \ref{fig:lt2} and let $J_t=\{a, b, c\}$.
Then $h(L)=4$ and $h(\ov L)=5$. In this case the height of the lattice
$\tilde L$ is $5$, with a maximal chain 
\[
X>\cO^X(d)>J_t>\cO^X(f)\cap J_t > \cO^X(f)\cap J_t\cap \cO^X(a)>\emptyset.
\]
\end{expl}
\begin{expl}\label{ex:lheight2}
Let $\G$ be the graph obtained by removing vertex $t$ from the
graph $\ov\G=\G_2$ of Figure \ref{fig:lt2} and again let $J_t=\{a, b, c\}$.
Then $h(L)=5$ and $h(\tilde L)=5$. In this case the maximal chains
in the lattice $L$ involve only the vertices $g,h,i,j,k,l$ and
the sets of $L_t$ involve only vertices $a,b, c$. Therefore the
lattice $\tilde L$ has some new chains of length $5$ but none of
length $6$. However computation shows (see \cite{DKR2}) that $h(\ov L)=6$.
\end{expl}
\begin{figure}
    \psfrag{a}{$a$}
    \psfrag{b}{$b$}
    \psfrag{c}{$c$}
    \psfrag{d}{$d$}
    \psfrag{e}{$e$}
    \psfrag{f}{$f$}
    \psfrag{g}{$g$}
    \psfrag{h}{$h$}
    \psfrag{i}{$i$}
    \psfrag{j}{$j$}
    \psfrag{k}{$k$}
    \psfrag{l}{$l$}
    \psfrag{x}{$t$}
    \mbox
    {
      \parbox[b]{.3\textwidth}
      {
        \centerline{
          \includegraphics[scale=0.4]{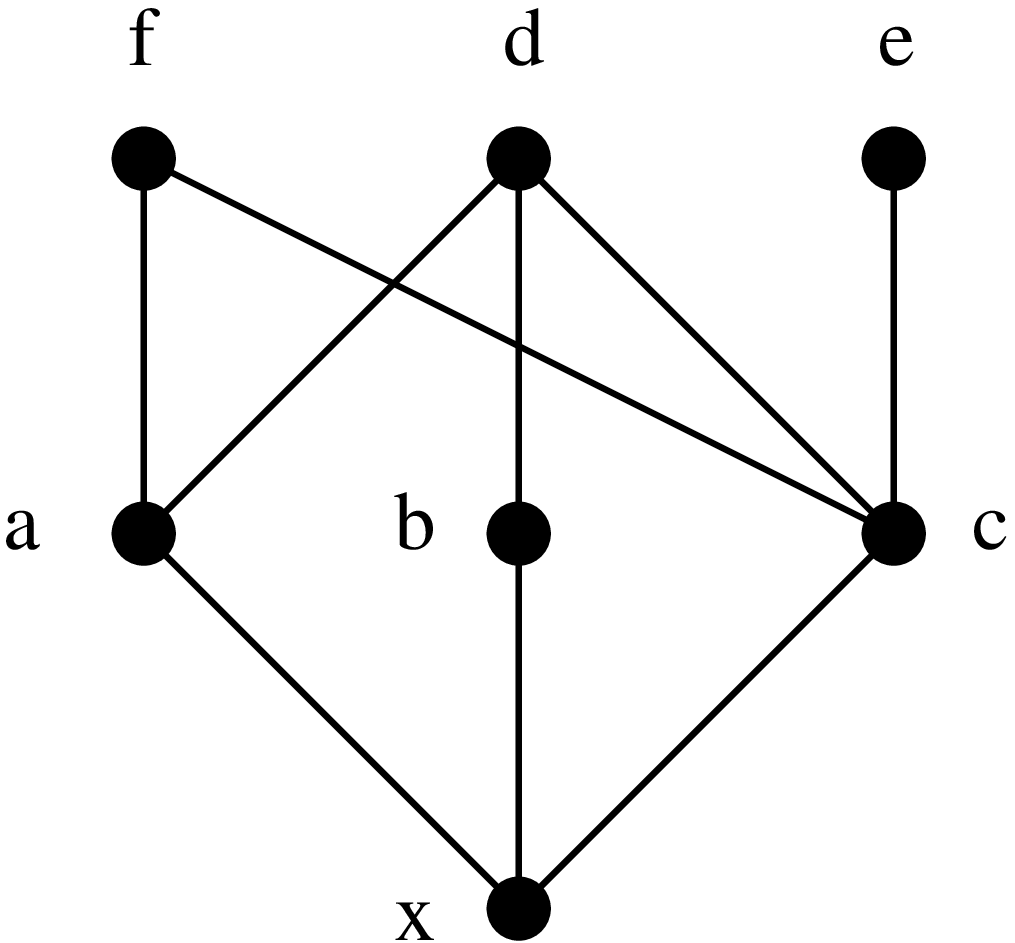} \\[.5em]%
                  }
        \centerline{$\G_1$}
      }
      \qquad
      \parbox[b]{.6\textwidth}
      {
        \centerline{
          \includegraphics[scale=0.4]{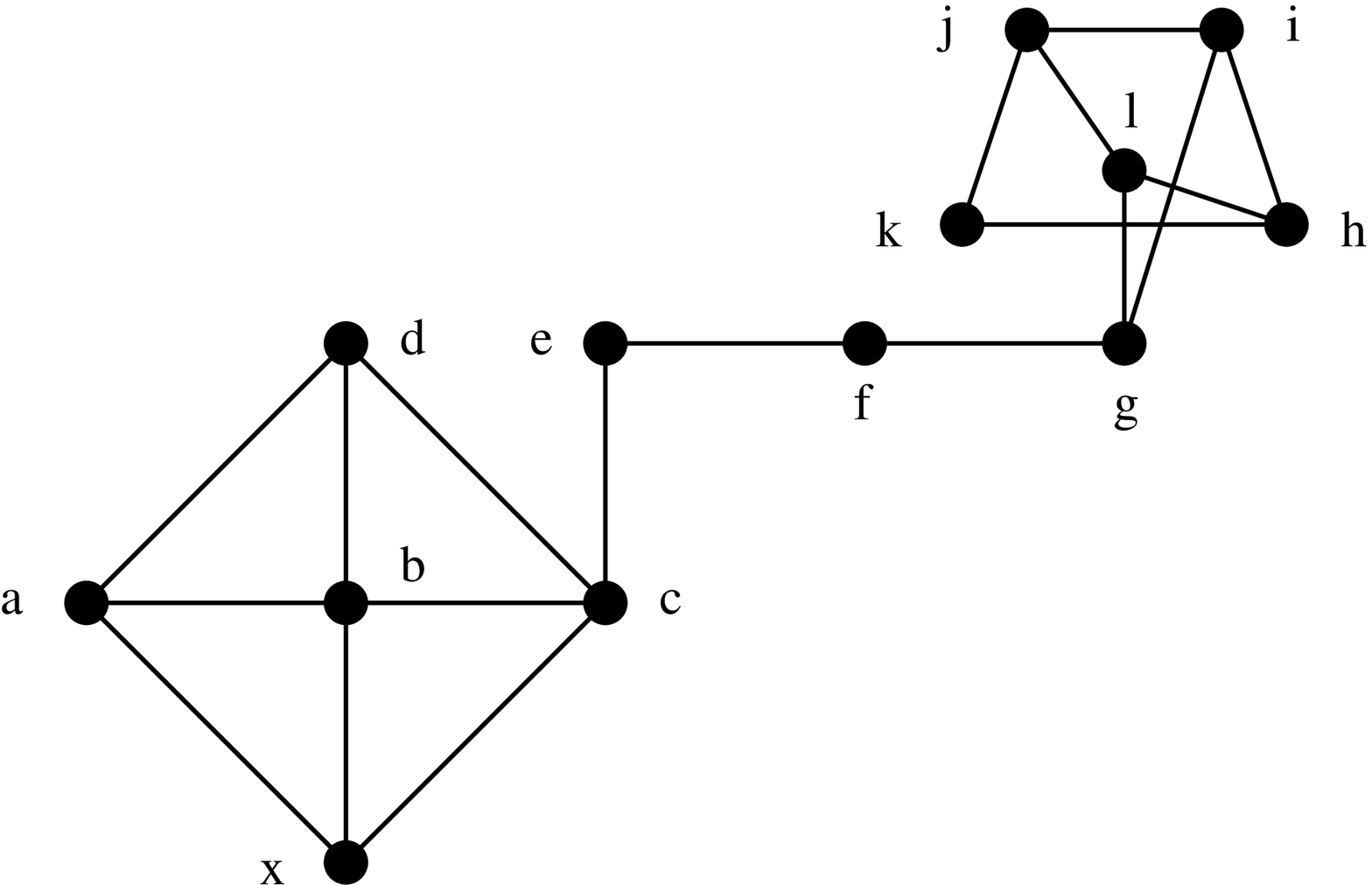}\\[.5em]%
                  }
        \centerline{$\G_2$}
      }
    }
  \caption{Examples \ref{ex:lheight1} and \ref{ex:lheight2}}\label{fig:lt2}
\end{figure}

\subsection{The structure of the extended lattice}\label{sec:SE}

Next we use the results of the Section \ref{sec:av} to describe the
lattice $\ov L$ in terms of the lattice $L$.
We make the following
definition.
Suppose that $L$ is a lattice which is a subset of a
lattice $L^\prime$ and that
the partial ordering in $L$
is the restriction of the partial ordering in $L^\prime$.
Assume that $L$ contains a subset $S$ such that there is
an isomorphism of partially ordered sets, $\rho$,  from $S$ to
$L^\prime \backslash L$.
Then we say
that $L^\prime$ is obtained from $L$ by {\em doubling} $S$ {\em along}
$\rho$.

Recall from Section \ref{sec:av} that if $Z\in \tilde L$ and $Z\nsubseteq J_t$ then
$Z\in L$. This, together with \eqref{eq:L_t}, prompts the following definition.
\begin{defn}
Let
\[
R=\{Z\in \tilde L|Z\nsubseteq J_t \textrm{ and } \cl(Z\cap J_t)=Z\}
\]
and let $\rho$ be the map from $R$ to $\tilde L$ given by
$\rho(Z)=Z\cap J_t$.
\end{defn}
If $Z\in R$ then $Z\in L$ and $Z\notin L_t$, as $Z\nsubseteq J_t$.
Furthermore, from \eqref{eq:L_t},
$\rho(Z)\in L_t\backslash L=\tilde L\backslash L$.
\begin{prop}\label{prop:LtoiL}
$\tilde L$ is obtained from
$L\subseteq \tilde L$ by doubling $R$ along $\rho$.
\end{prop}
\begin{proof}
As $\rho$ clearly preserves inclusion it suffices to show that $\rho$ is
a bijection. If $\rho(Y)=\rho(Z)$, with $Y,Z\in R$ then
$Z=\cl(Z\cap J_t)=\cl(Y\cap J_t)=Y$, so $\rho$ is injective.
From \eqref{eq:L_t} if follows that $\rho$ is also surjective.
\end{proof}

The lattice $\ov L$ is obtained from $\tilde L$ by a doubling on an
appropriate subset of $\tilde L$. To see this we use the following
strengthening of the final part of Lemma \ref{lem:ovbg}.
We remark that  condition \eqref{eq:YandYt} of the lemma can be expressed more succinctly
in terms of complements by noting that
\be
\item $\cO^X(\cO^X(Y)\cap J_t)=\cO^X(\cO^{J_t}(Y))$ and
\item if $Y\subseteq J_t$ then $(\cO^X(\cO^X(Y)\cap J_t))\cap J_t=\cl^{J_t}(Y)
\in \CS(J_t)$.
\ee
\begin{lem}\label{lem:ovLextra}
Let $Y\subset X$. Then $Y$ and $Y\cup \{t\}$ belong to $\ov L$ if and
only if $\cO^X(Y)\nsubseteq J_t$ and
\begin{equation}\label{eq:YandYt}
Y=
\left\{
\begin{array}{ll}
\cO^X(\cO^X(Y)\cap J_t), &\textrm{ if } Y\nsubseteq J_t\\
\cO^X(\cO^X(Y)\cap J_t)\cap J_t, &\textrm{ if } Y\subseteq J_t
\end{array}
\right.
.
\end{equation}
\end{lem}
\begin{proof}
Suppose that
$\cO^X(Y)\nsubseteq J_t$. 
If $Y\nsubseteq J_t$ then
\begin{align}\label{al:YninJ}
\cl^{\ov X}(Y\cup \{t\})
&=\cO^{\ov X}( \cO^X(Y)\cap J_t)\notag\\
&=\cO^{X}( \cO^X(Y)\cap J_t)\cup \{t\}.
\end{align}

If, on the other hand,  $Y\subseteq J_t$ then
$\cO^{\ov X}(Y\cup \{t\})=(\cO^X(Y)\cap J_t)\cup \{t\}$ so
\begin{align}\label{al:YinJ}
\cl^{\ov X}(Y\cup \{t\})
&= \cO^{\ov X}\left((\cO^X(Y)\cap J_t)\cup \{t\} \right)\notag\\
&=\left(\cO^{X}( \cO^X(Y)\cap J_t)\cup \{t\}\right)\cap (J_t\cup \{t\})\notag\\
&=\left(\cO^{X}( \cO^X(Y)\cap J_t)\cap J_t\right)\cup \{t\}.
\end{align}
In both cases, if in addition \eqref{eq:YandYt} holds then
$\cl^{\ov X}(Y\cup \{t\})=Y\cup \{t\}$ and
$Y\cup \{t\}\in \ov L$.

Now, given that $\cO^X(Y)\nsubseteq J_t$ and \eqref{eq:YandYt} holds, choose
$x\in \cO^X(Y)$ such that $x\notin J_t$. Then
$\cO^{\ov X}(x)=\cO^X(x)\supseteq \cl^X(Y)\supseteq Y$
and $t\notin \cO^{\ov X}(x)$.
From the above $Y\cup \{t\}\in \ov L$, so $Y\cup \{t\}=\cO^{\ov X}(Z)$,
for some $Z\in \ov L$. Then $\cO^{\ov X}(Z\cup \{x\})=
\cO^{\ov X}(Z)\cap \cO^{\ov X}(x)=Y$; and $Y\in \ov L$.

Conversely suppose that $Y$ and $Y\cup \{t\}$ belong to $\ov L$.
In this case if  $\cO^X(Y)\subseteq J_t$ then
$\cO^{\ov X}(Y)\subseteq \cO^X(Y)\cup \{t\}$
so $\cl^{\ov X}(Y)\supseteq
\cO^{\ov X}( \cO^X(Y))\cap (J_t \cup\{t\})$.
Thus $t\in \cl^{\ov X}(Y)$ and
$Y\notin \ov L$, a contradiction. Thus $\cO^X(Y)\nsubseteq J_t$.
If $Y\nsubseteq J_t$ then, from \eqref{al:YninJ},
$Y=\cl^{\ov X}(Y\cup \{t\})\backslash \{t\}
=\cO^X (\cO^X(Y)\cap J_t)$.
If, on the other hand, $Y\subseteq J_t$ then
\eqref{al:YinJ} implies that
$Y=(\cO^{X}( \cO^X(Y)\cap J_t))\cap J_t$, as claimed.
\end{proof}
The lemma prompts the following definition.
\begin{defn}\label{def:extset}
Let
\[
S_1=\{Y\subset X| Y\nsubseteq J_t, \cO^X(Y)\nsubseteq J_t, \textrm{ and }
Y=\cO^X (\cO^X(Y)\cap J_t)\}
\]
and
\[
S_2=\{Y\subset X| Y\subseteq J_t, \cO^X(Y)\nsubseteq J_t, \textrm{ and }
Y=(\cO^{X}( \cO^X(Y)\cap J_t))\cap J_t\}.
\]
Let $S=S_1\cup S_2$ and let $T=\{Y\cup \{t\}|Y\in S\}$.
Let $\s$ be the map from $S$ to
$T$ given by $\s(Y)=Y\cup \{t\}$.
\end{defn}
From Lemma \ref{lem:ovLextra} it follows that $S\cup T\subseteq \ov L$ 
and by definition $S\subseteq \tilde L$. Moreover, from
Lemma \ref{lem:ovbg},  $\ov \b(Y)=Y$, for 
all $Y\in S$, so $S=\ov \b(S)\subseteq \ov \b(\tilde L)\subseteq \ov L$.
\begin{prop}\label{prop:ovLdouble}
The lattice $\ov L$ is obtained from $\ov \b(\tilde L)\subseteq \ov L$ by
doubling $S$ along $\s$.
\end{prop}
\begin{proof} Using Lemma \ref{lem:ovbg}, 
if $Y\in S$ and $Y\cup \{t\}=\ov \b(U)$, for some element $U\in \tilde L$, then
$t\in \ov \b(U)$ implies that $\cO^X(U)\in J_t$. However $Y\cup \{t\}=\ov \b(U)
=U\cup \{t\}$ so $U=Y$ and $\cO^X(Y)\nsubseteq J_t$, a contradiction. Hence
no element of $T$ belongs
to the image of $\ov \b$.
If $Z\in \ov L$ and $Z$ is not in the image of $\ov \b$
then, from Lemma \ref{lem:ovbg} again,  $\ov \g(Z)=Z\backslash \{t\}\in \tilde L$ and
so $\ov \b(Z\backslash \{t\}) \neq Z$.  
Thus either $t\notin Z$ and $\ov \b(Z\backslash \{t\})=\ov \b(Z)=Z\cup \{t\}$ 
or $t\in Z$ and $\ov \b(Z\backslash \{t\})=Z\backslash \{t\}$. In the former
case $Z\in \ov L$ and $\ov \b(Z)=Z\cup \{t\}\in \ov L$ so $Z\in S$ and 
$Z\cup \{t\}\in T\cap \operatorname{Im}(\ov \b)$, 
a contradiction. 
Hence $\ov \b(Z\backslash \{t\})=
Z\backslash \{t\}\in \ov L$ and  $t\in Z$. It follows from Lemma \ref{lem:ovLextra}
that $Z\backslash \{t\}\in S$ so $Z\in T$. That is, $T=
\ov L\backslash \ov \b(\tilde L)$. As $\s$ is an
inclusion preserving bijection the result follows.
\end{proof}
\subsection{Extension along the complement of a simplex}\label{subs:simp}
In those cases where $\g$ is injective it follows, from Corollary \ref{cor:gb}, that
$\g$ is a bijection and so an isomorphism of lattices.  We now consider
under which conditions this may occur.
Let $V=\cO^{\ov X}(t)=J_t\cup \{t\}\in \ov L$.
If $\g$ is injective
then $V=\b\g(V)=\cl^X(J_t)\cup \{t\}$, so $J_t=\cl^X(J_t)\in L$. Therefore
$J_t\in L$ is a necessary condition for $\g$ to be injective. We shall
show, in Section \ref{sec:EC}, that if $J_t$ is closed then $h(L)=h(\ov L)$;
but
we shall also see in 
Lemma  \ref{lem:ginj} that a further condition is required to ensure that
$\g$ is injective. 
First however we establish a simple form for $\g$ when $J_t$ is
closed.
\begin{lem}\label{lem:gsimple}
If $J_t\in L$ then $\tilde L=L$ and $\g(Z)=Z\backslash\{t\}$, for all $Z\in {\ov L}$.
Moreover, in the notation of {\rm Definition \ref{def:extset}}, $S_1=\emptyset$ so $\ov L$
is obtained from $\b(L)\subseteq \ov L$ by doubling $S_2$ along $\s$.
\end{lem}
\begin{proof}
If $J_t\in L$ then $L_t$ is a subset of $L$, so $\tilde L=L$, as claimed.
In this case $\g=\ov \g$ and $\b=\ov \b$, so the first statement of the
Lemma follows from Lemma \ref{lem:ovbg}. If $Y\in S_1$ then $Y\in L$ and
$Y=\cO^X(W)$, where $W=\cO^X(Y)\cap J_t\in L$. However this means
$\cO^X(Y)=W\subseteq J_t$, a contradiction.
\end{proof}
\begin{lem}\label{lem:ginj}
The map $\g$ is an isomorphism of lattices if and only if
$J_t=\cO^X(S)$, where $S$ is a simplex of $\G$.
\end{lem}
\begin{proof}
First assume that  $J_t=\cO^X(A)$, where $A\subseteq X$ is a simplex.
In this case, in the notation of Definition \ref{def:extset},
$Y\in S_2$ implies $Y\in \CS(J_t)$, so $Y=\cO^{J_t}(W)$, for some
$W\subseteq J_t$. Now $W\subseteq J_t=\cO^X(A)=\cO^{J_t}(A)$
which implies $\cO^{J_t}(\cO^{J_t}(A))\subseteq \cO^{J_t}(W)=Y$.
As $A$ is a simplex $A\subseteq J_t$ so $A\subseteq \cO^{J_t}(\cO^{J_t}((A))$ and thus
$\cO^X(Y)\subseteq J_t$, contrary to the definition of $S_2$. Therefore
$S_1=S_2=\emptyset$ and from Lemma \ref{lem:gsimple} $L=\ov L$.

On the other hand suppose that $J_t=\cO^X(N)$, where $N$ is not a simplex.
Then, from Lemma \ref{lem:orth}.\ref{it:orth6}, there is $s\in N$ such that
$s\notin J_t$. Therefore $t\notin \cO^{\ov X}(N)$ and we have $J_t=\cO^{\ov X}(N)
\in \ov L$. Hence $\g(J_t)=J_t=\g(J_t\cup \{t\})$ and $\g$ is not injective.
From the remarks at the beginning of the Section it follows that if $J_t$ is
not the orthogonal complement of a simplex in $X$ then $\g$ is not injective.
(It is not difficult to see that in this case $J_t\in S_2$.)
\end{proof}
As a consequence of this lemma we obtain the following theorem.
\begin{thm}\label{thm:ginj}
The lattices $L$ and $\ov L$ are isomorphic if and only if
$J_t=\cO^X(S)$, where $S\subset X$ is a simplex,  in which case $\g$ is an isomorphism.
\end{thm}
\begin{proof}
From Lemma \ref{lem:ginj}, if $J_t=\cO^X(S)$, where $S\subset X$ is a simplex, then
 the lattices are isomorphic
and $\g$ is an isomorphism.
Now suppose that $J_t$ is not of this form. The map $\b:L\maps \ov L$ is injective
so $|L|\le |\ov L|$. If $|L|=|\ov L|$ then,
as
$\g\b=\id_L$,
it follows that $\g$ is also injective, contrary to  Lemma
\ref{lem:ginj}.
Thus $|L|<|\ov L|$ and the
lattices are not isomorphic.
\end{proof}
\subsection{Abelian Inflation and Deflation}\label{sec:abinf}
In this section we consider further the case  where the
set $J_t$ defined above is the orthogonal complement of a simplex, as in the previous section.
First
we introduce some equivalence classes on subsets of vertices $\G$.
We say that two subsets
$S$ and $T$
of $X$ are $\perp$-{\em equivalent} in $X$  and write
$S\sim_\perp T$ if and
only if $S^\perp=T^\perp$; that is  $\cO^X(S)=\cO^X(T)$.
\begin{lem} \label{lem:eq}
Let $S$ and $T$ be subsets of $X$.
\be
\item $S\sim_\perp T$ if and only if
$T\subseteq \cl^X(S)$ and $S\subseteq \cl^X(T)$.
    \item If $S\sim_\perp T$  and $Y\in \CS(\G)$ then $S\subseteq Y$
implies that $T\subseteq Y$.
\item \label{it:eq} If $S$ is a simplex and $S\sim_\perp T$ then $T$ is a simplex.
In particular, in this case,
$G(\G^\prime)$ is an Abelian group, where $\G^\prime$
denotes the full subgraph of $\G$ on $S\cup T$.
\ee
\end{lem}
\begin{proof}
To see the first statement note that, using Lemma \ref{lem:orth},
$S\sim_\perp T$ if and only if $\cl^X(S)=\cl^X(T)$. It follows that
$S\sim_\perp T$ implies that $S\subseteq \cl^X(T)$ and $T\subseteq \cl^X(S)$.
Conversely if $S\subseteq \cl^X(T)$ then $S^\perp \supseteq T^{\perp\perp\perp}
=T^\perp$. Similarly if $T\subseteq \cl^X(S)$ 
then $T^\perp \supseteq S^\perp$ and
the result follows.
 To prove the second
statement note that by Lemma \ref{lem:cl}, $S\subseteq Y$ and $Y$ closed
implies $\cl^X(S)\subseteq Y$. Thus $T\subseteq \cl^X(T)=\cl^X(S)\subseteq Y$.
For the third statement we have $S\subseteq \cO^X(S)=\cO^X(T)$, since $S$ is
a simplex, and so $T\subseteq \cO^X(S)=\cO^X(T)$. Hence $T$ is a simplex
and the result follows.
\end{proof}
In the light of Lemma \ref{lem:eq}.\ref{it:eq} we define the {\em
Abelian closure} $\acl(S)$ of a simplex $S$ to be the union of  subsets $T$ of
$X$ such that $S\sim_\perp T$. Then $S\subseteq \acl(S)$ and
it is easy to see then that $\acl(S)$ is
the unique maximal simplex such that $S\sim_\perp \acl(S)$.

Now let $\D$ be a graph with vertices $V$. Let $S$ be a simplex of
$\D$ and $y\in V$ with
$y\notin S$ and
 suppose that $S\sim_\perp \{y\}$ in $\D$: that is $\cO^{V}(S)
=\cO^{V}(y)$.  Let $\D_y=\D\backslash \{y\}$
. Then $\D_y$
is called an {\em elementary Abelian deflation} of
$\D$ and $\D$ is called an {\em elementary Abelian inflation}
of $\D_y$. In this case the subgroup of $\D_y$ generated by
$S$ is a free Abelian group of rank $|S|$ and the subgroup
of $\D$ generated by $S\cup \{y\}$ is free Abelian of rank $|S|+1$.

If a graph $\Omega$ can be obtained from a graph
$\Gamma$ by finitely many elementary Abelian inflations then
$\Omega$ is called an {\em Abelian inflation} of $\Gamma$ and $\Gamma$ is
called an {\em Abelian deflation} of $\Omega$. The same terminology
carries over to the respective partially commutative groups.

\begin{prop}\label{prop:abinf}
If $\D$ is an Abelian inflation of $\G$ then $\CS(\D)\simeq \CS(\G)$.
\end{prop}
\begin{proof}
It suffices to prove the result in the case where $\D$ is an
elementary Abelian inflation of $\G$. Suppose then that
$\G=\D_t$, for some vertex $t$ of $\D$. To be more explicit let
$V(\D)=\ov X$, assume
that $t\in \ov X$,  $S\subseteq \ov X$ is a simplex,  $t\notin S$ and
$S\sim_\perp \{t\}$ in $\D$.
Let $X=V(\G)$. Then, as $\G=\D_t$ we have $\ov X=X\cup \{t\}$
and $\cO^{\ov X}(t)=\cO^{\ov X}(S)$.
Let $J_t=\cO^{\ov X}(t)\backslash\{t\}$. Then, as $S\subseteq X$, we
have $\cO^X(S)=J_t\in \CS(\G)$.
As $\D$ is obtained from $\G$ by adding the vertex $t$ which is joined
to precisely those vertices in $J_t=\cO^X(S)$, and $S$
is a simplex,
it follows from Theorem \ref{thm:ginj}
that $\CS(\G)\simeq\CS(\D)$, as claimed.
\end{proof}

\subsection{Extension along a closed set}\label{sec:EC}
We saw in Section \ref{subs:simp} that if $J_t$ is a closed set then,
in the notation of Definition \ref{def:extset}, $\ov L$ is obtained from $\b (L)$
by doubling $S_2$ along $\s$.
In this section we shall show that if $J_t$ is closed then
$h(\ov L)= h(L)$. If $J_t=\cO^{X}(S)$ where $S$ is a simplex then
$\ov \G$ is an Abelian inflation of $\G$, so this follows from
Proposition \ref{prop:abinf}. Therefore we assume that
$A\subseteq X$, such that $A$ is not a
simplex, and $J_t=\cO^{X}(A)\in L$. As $A$ is not a simplex
the set $A^\prime = A\backslash J_t$ is non-empty. Fix $a\in A^\prime$.

Now let $Y\in \overline L$ with $t\in Y$. Then $Y=\cO^{\overline X}(Z)$,
where $Z\subseteq J_t\cup \{t\}$.
There are two possibilities. Either
\be
\item
$Z\subseteq J_t$, in which case $A\cup \{t\}\subseteq \cO^{\overline X}(Z)=Y$; or
\item
$Z\nsubseteq J_t$, in which case $Z=W\cup \{t\}$, where $W\subseteq J_t$, so $a\notin Y$.
\ee
In the latter case
\begin{align*}
Y&=\cO^{\overline X}(W\cup \{t\})\\
&= (\cO^{X}(W)\cup \{t\}) \cap (J_t\cup \{t\})\\
&=(\cO^{X}(W)\cap J_t) \cup  \{t\}
\end{align*}
whereas
\begin{align*}
\cO^{\overline X}(W\cup \{a\})
&= \cO^{\overline X}(W)\cap \cO^{\overline X}(a)\\
&= (\cO^{X}(W)\cup \{t\}) \cap \cO^{X}(a)\\
&=\cO^{X}(W)\cap \cO^{X}(a).
\end{align*}
This prompts us to
define a map $\a:\ov{L}\maps \ov{L}$ by
\[\a(Y)=\left\{
\begin{array}{l}
\cO^{\overline X}(W\cup \{a\}) \textrm{ if } t\in Y, a\notin Y
\\
Y \textrm{ otherwise}
\end{array}
.\right.
\]
Note that
\begin{equation}\label{eq:Ysub}
t\notin \a(Y) \textrm{ and }
Y\backslash \{t\}\cup \{a\}\subseteq \a(Y),\textrm{ if } t\in Y
\textrm{ and } a\notin Y
\end{equation}
and that
\begin{equation}\label{eq:Yalt}
\textrm{ either } t\notin \a(Y) \textrm{ or } A\cup \{t\}\subseteq \a(Y)
\textrm{ for all } Y\in \ov L.
\end{equation}
Now let $\ov \cC=Z_1<\cdots <Z_k$ be a strong ascending chain in $\ov L$.
Let $\a(\ov \cC)= \a(Z_1)\le \cdots \le \a(Z_k)$.
\begin{lem}\label{lem:alpchain}
$\a(\ov \cC)$ is a strong ascending chain in $\ov L$.
\end{lem}
\begin{proof}
Define
$r=r(\ov \cC)$ to be the smallest integer such that $t\in Z_r$.
If no such $r$ exists then $\a(\ov \cC)=\ov \cC$ and there is nothing to
prove. Suppose then that $1\le r\le k$.
Let $s$ be the smallest integer such that $A\cup\{t\}\subseteq Z_s$ (and
set $s=k+1$ if  $A\cup\{t\}\nsubseteq Z_k$). Then
$r\le s\le k+1$. For $i$ such that $1\le i\le r-1$ or  $s\le i \le k$ we
have $\a(Z_i)=Z_i$. Therefore we need only check that $\a(Z_i)<\a(Z_{i+1})$ for
$i$ such that $r-1\le i\le s$. If $r=s$ then also $\a(Z_r)=Z_r$ and so
$\a(\ov \cC)=\ov\cC$ and the Lemma holds.

Assume then that $r<s$. In this case $a\notin Z_r$ and so $a\notin Z_{r-1}$.
Therefore $a\notin \a(Z_{r-1})=Z_{r-1}$ but $a\in \a(Z_r)$. As $t\notin Z_{r-1}$ we
have
\[Z_{r-1}\le Z_r\backslash \{t\}\le \a(Z_r),\]
so $\a(Z_{r-1})<\a(Z_r)$.

To see that $\a(Z_{s-1})<\a(Z_s)$ write
$Z_s=\cO^{\ov X}(Y_s)$, where $Y_s\subseteq J_t$ and
$Z_{s-1}=\cO^{\ov X}(W_{s-1}\cup \{t\})$, where $W_{s-1}\subseteq J_t$.
As $Z_{s-1}<Z_s$ we have $W_{s-1}\cup \{t\}\ge Y_s$ and, as $t\notin Y_s$,
$W_{s-1}\ge Y_s$; so $\cO^X(W_{s-1})\le \cO^X(Y_s)$. Therefore
\[\a(Z_{s-1})= \cO^X(W_{s-1})\cap \cO^X(a)\le \cO^X(Y_s)< \cO^X(Y_s)\cup\{t\}=Z_s=\a(Z_s).\]

It remains to check that $\a(Z_i)<\a(Z_{i+1})$, where $r\le i \le s-2$. Given such $i$
we have, for $j=i$ and $j=i+1$,
\[Z_j=\cO^{\ov X}(W_j\cup \{t\})=
(\cO^{X}(W_j)\cap J_t) \cup  \{t\},\]
where $W_j\subseteq J_t$.
As $Z_i<Z_{i+1}$
we have $W_i>W_{i+1}$ so $\cO^X(W_i)\le \cO^X(W_{i+1})$. Therefore
\[\a(Z_i)=\cO^X(W_i)\cap \cO^X(a)\le  \cO^X(W_{i+1})\cap \cO^X(a)=\a(Z_{i+1}).\]
Moreover, as $Z_i<Z_{i+1}$
there is $x\in \cO^{X}(W_{i+1})\cap J_t$ such that
$x\notin \cO^{X}(W_i)\cap J_t$. Hence $x\notin\cO^{X}(W_i)$ and therefore
$x\notin \a(Z_{i})$.  However  $J_t\subseteq \cO^X(a)$ so $x\in \cO^{X}(W_{i+1})\cap J_t$
implies $x\in \a(Z_{i+1})$. Thus $\a(Z_i)<\a(Z_{i+1})$.
\end{proof}

Given a chain $\ov \cC= Z_1<\cdots <Z_k$ in $\ov L$ define $\g(\ov \cC)$
to be the chain $\g(Z_1)\le \cdots \le \g(Z_k)$.
%
\begin{lem}\label{lem:geasy}
If $\ov \cC$ is a strictly ascending chain in $\ov L$ such that $Z_i$ satisfies
\eqref{eq:Yalt}, for $i=1,\ldots, k$,
then $\g(\ov \cC)$ is a strictly ascending chain in $L$.
\end{lem}
\begin{proof}
As before define $r=r(\ov \cC)$ to be the smallest integer such that $t\in Z_r$.
As $J_t$ is closed we have $\g(Z)=Z\backslash \{t\}$, for all $Z\in \ov L$.
Therefore it suffices to show that $\g(Z_{r-1})<\g(Z_r)$. We have $A\subseteq
\g(Z_r)$, by \eqref{eq:Yalt}. If $A\subseteq Z_{r-1}=\cO^{\ov X}(Y_{r-1})$ then
$\cO^{\ov X}(t)=J_t\cup \{t\}=\cO^{\ov X}(A)\supseteq Y_{r-1}$. In this case
$t\in \cl^{\ov X}(t)\subseteq Z_{r-1}$, contrary to the definition of $r$.
Hence $A\nsubseteq Z_{r-1}$ and so $\g(Z_{r-1})<\g(Z_r)$.
\end{proof}

Now, given any strictly ascending chain $\ov \cC$ in $\ov L$ we may, according to
Lemma \ref{lem:alpchain},
construct a strictly ascending chain $\cC=\a(\ov \cC)$, satisfying
\eqref{eq:Yalt}; as in the hypothesis of Lemma \ref{lem:geasy}.
Applying $\g$ to $\cC$ we obtain a strictly ascending chain
$\g(\cC)$ in $L$ of the same length as $\ov \cC$. Therefore we have the following
proposition.
\begin{prop}
If $J_t$ is closed then $h(L)=h(\ov L)$.
\end{prop}
\subsection{Extension along the complement of a co-simplex}\label{sec:ECS}
A subset $A\subseteq X$ is called a {\em co-simplex} if $A\cap \cO^X(A)=\emptyset$.
In this section we consider the case $J_t=\cO^X(A)$ where $A$ is a co-simplex.
In this case if $Y\in L$ and $Y\subseteq J_t$ then $\cO^X(Y)\supseteq \cO^X(J_t)\supseteq
A$. As $A\cap J_t=\emptyset$ we have $\cO^X(Y)\nsubseteq J_t$, for all such $Y$.
Therefore, if $A$ is a co-simplex,
\[
S_2=\{Y\in L|Y\subseteq J_t,Y\in \CS(J_t)\}=L\cap \CS(J_t)=\CS(J_t),
\]
as $\CS(J_t)\subseteq L$. Therefore $\ov L$ is obtained from $L$ by doubling
$\CS(J_t)$ along $\s$.

It is easy to find examples showing that in general there may be elements of
$L$ which are subsets of $J_t$ but do not  belong to
$\CS(J_t)$. This motivates the following definition.
\begin{defn}\label{def:realisable}
A closed subset $J\in L$ is {\em realisable} if
$\CS(J)=\{Y\in L|Y\subseteq J\}$.
\end{defn}
\begin{lem}\label{lem:realisable}
An element $J\in L$ is realisable if and only if,
for all $s\in X\backslash J$ there exists $W\subseteq J$
such that $\cO^X(s)\cap J=\cO^X(W)\cap J$.
\end{lem}
\begin{proof}
Let $J=\cO^X(A)$, where $A\in L$.
Suppose that $J$ is realisable and that $s\in X\backslash J$.
Then $Y=\cO^X(s)\cap J\in L$ and $Y\subseteq J$; so $Y\in \CS(J)$. Hence
$Y=\cO^X(W)\cap J$, where $W=\cO^X(U)\cap J$, for some $U\subseteq J$,
so $W\subseteq J$ as required.

Now suppose that $J$ satisfies the condition of the Lemma. Let $Y\in L$ such that
$Y\subseteq J$. Then $Y=\cO^X(Z)$, for some $Z\in L$. Let $Z_1=Z\cap J$ and
$Z_2=Z\backslash Z_1$. Fix $z\in Z_2$. By hypothesis there exists
$W_z\subseteq J$ such that $\cO^X(z)\cap J=\cO^X(W_z)\cap J\in \CS(J)$.
Therefore
\[
\cO^X(Z_2)\cap J=\bigcap_{z\in Z_2}(\cO^X(z)\cap J)\in \CS(J).
\]
As $Z_1\subseteq J$ it is also true that $\cO^X(Z_1)\cap J\in \CS(J)$.
We have $Y\subseteq J$ so
$Y=\cO^X(Z)\cap J=(\cO^X(Z_1)\cap J)\cap (\cO^X(Z_2)\cap J)\in \CS(J)$,
as required.
\end{proof}
We now have the following proposition.
\begin{prop}
Let $A$ be a co-simplex such that $\cO^X(A)$ is realisable. If $J_t=\cO^X(A)$
then $\ov L$ is obtained from $L$ by doubling
$S_2=\{Y\in L|Y\subseteq J_t\}$ along $\s$.
\end{prop}
\subsection{Free Inflation and Deflation}\label{sec:FI}
We now define another relation on the subsets of $X$, similar to
that of Section \ref{sec:abinf}: but giving rise to free groups instead
of free Abelian groups.
If $J_t$ is the orthogonal complement of a co-simplex $A$ then
$\cO^{\ov X}(A)=J_t$, since $A\cap J_t=\emptyset$, so
$\cO^{\ov X}(A)\backslash A=\cO^{\ov X}(t)\backslash \{t\}$.
This suggests the following definition.
If $Y$ and $Z$ are subsets of $X$ we say that
$Y$ and $Z$ are $o${\em -equivalent} and write
$Y\sim_o Z$ if
\begin{equation} \label{eq:confinf}
\cO^X(Y) \backslash Y =\cO^X(Z) \backslash Z.
\end{equation}
Note that if $Y$ is a co-simplex then
$Y\sim_o Z$ implies that
$\cO^X(Y,Z)=\cO^X(Y)$ and that
$G(\cO^X(Y)) = C(Y\cup Z)$
--- the centraliser of $Y$ and $Z$ in $G(\Gamma)$.

We call a co-simplex $A$ such that the full subgraph on $A$ is
the null graph a {\em free} co-simplex.
If $A$ is a free co-simplex and $B$ is either a free co-simplex
or a single vertex then $A\sim_o B$ implies that the subgroup
of $G$ generated by $A$ and $B$ is a free group. We define the
{\em free-closure} $\fcl(A)$ of  a free co-simplex $A$ to be the
union of all free co-simplexes $B$ such that $A\sim_o B$. It is easy
to see that $A\subseteq \fcl(A)$ and that
$\fcl(A)$ is the unique maximal free co-simplex such that
$A\sim_o \fcl(A)$.

If $J_t$ is the orthogonal complement
of  a free co-simplex then we say that $\ov \G$ is an
{\em elementary free inflation} of $\G$ and that $\G$ is an 
{\em elementary free deflation} of $\ov \G$. We say that $\D$ is 
a
{\em free inflation} of $\G$ and  $\G$ is a 
{\em free deflation} of $\D$ if $\D$ is obtained from $\G$ by a finite
sequence of elementary free inflations.
\subsection{The Compression of a Graph}
We now use the ideas of Sections \ref{sec:abinf} and \ref{sec:FI} to define
an equivalence relation on the vertices of a finite graph $\G$; which
will give a decomposition of the automorphism group of $\G$. We build this
equivalence relation up out of the restrictions to singleton sets of the
relations defined in
Sections \ref{sec:abinf} and \ref{sec:FI}.

The restriction of the relation of $\perp$-equivalence to one-element subsets
of $X$ gives and equivalence relation $\sim_\perp$ on $X$ such that
$x\sim_\perp y$ if and only if $x^\perp=y^\perp$. Denote the equivalence
class of $x$ under $\sim_\perp$ by $[x]_\perp$.

Similarly, restricting the relation of $o$-equivalence to one-element subsets
gives an equivalence relation $\sim_o$ on $X$ such that $x\sim_o y$ if and
only if $x^\perp\backslash\{x\}=y^\perp\backslash\{y\}$.
Denote the equivalence
class of $x$ under $\sim_o$ by $[x]_o$.
\begin{lem}\label{lem:equiv}\
\be
\item $[x]_\perp$ is a simplex, for all $x\in X$.
\item $[x]_\perp\cap [x]_o=\{x\}$, for all $x\in X$.
\item If $|[x]_\perp|\ge 2$ then $| [x]_o|=1$.
\item If $|[x]_o|\ge 2$ then $[x]_o$ is a free co-simplex and $|[x]_\perp|=1$.
\ee
\end{lem}
\begin{proof}
If $x\sim_\perp y$ then, as $x\in x^\perp$ we have  $x\in y^\perp$, so
$[x]_\perp$ is a simplex. If $z\in [x]_\perp\cap [x]_o$ then $x\in
x^\perp =z^\perp$ but
$x\notin x^\perp\backslash\{x\}=z^\perp\backslash\{z\}$: so it must be
that $x=z$. If $|[x]_\perp|\ge 2$ then suppose that $y\neq x$ and $y\sim_\perp
x$. If $z\neq x$ and $z\sim_o x$ then $z\neq y$, as $[x]_\perp\cap [x]_o=\{x\}$
from the above. Thus $y\in y^\perp=x^\perp$ implies $y\in
x^\perp\backslash \{x\}=z^\perp\backslash\{z\}$, so $z\in y^\perp=x^\perp$,
contradicting $z\sim_o x$. A similar argument shows that if
$|[x]_o|\ge 2$ then  $|[x]_\perp|=1$. If $y\neq x$ and $y\sim_o x$ then
$y\notin x^\perp$, as otherwise $ x^\perp\backslash\{x\}\neq
y^\perp\backslash\{y\}$. Hence $[x]_o$ is a free co-simplex if
$|[x]_o|\ge 2$.
\end{proof}

Now define a relation $\sim$ on $X$ by
$x\sim y$ if and only if either $x\sim_\perp y$ or $x\sim_o y$. From Lemma
\ref{lem:equiv} $\sim$ is an equivalence relation and we denote the equivalence
class of $x$ under $\sim $ by $[x]$. Define subsets $M_1$, $M_\perp$ and
$M_o$ of $X$ by
\begin{align*}
M_1&=\{x\in X: [x]=[x]_o=[x]_\perp=\{x\}\},\\
M_\perp &= \{x\in X: |[x]_\perp| \ge 2\} \textrm{ and }\\
M_o&= \{x\in X: |[x]_o| \ge 2\}.
\end{align*}
From Lemma \ref{lem:equiv} it follows that $X$ is the disjoint
 union $X=M_1\sqcup M_\perp\sqcup M_o$.

We use the equivalence $\sim$ to define a quotient graph of $\G$.
\begin{defn}\label{def:comp}
The {\em compression} of the graph $\G$ is the graph $\G^{\tc}$ with
vertices $X^\tc=\{[v]:v\in X\}$ and an edge joining
$[u]$ to $[v]$ if and only if $(u^\prime,v^\prime)$ is an edge of
$\G$ for all
$u^\prime\in [u]$ and $v^\prime \in [v]$.
\end{defn}
Note that although $\G$ has no loops it may be that there are loops
in $\G^\tc$ (if there are vertices of $\G$ such that $[x]_\perp$ has more
than two elements). If $\G$ and $\G^\prime$ are graphs without multiple
edges, and
there is a map $f:V(\G)\maps V(\G^\prime)$ then we say that $f$ induces
a graph  homomorphism $f:\G\maps \G^\prime$ if $(f(u),f(v))\in E(\G)$
for all $(u,v)\in \G$.
\begin{prop}\label{prop:comp}
The map $\tc:X\maps X^\tc$ given by $\tc(x)=[x]$, for $x\in X$, induces
a surjective graph homomorphism $\tc:\G\maps \G^\tc$.
\end{prop}
\begin{proof}
The map $\tc:X\maps X^\tc$ is surjective by definition. If $\tc$ maps
edges of $\G$ to edges of $\G^c$ then,
since neither graph has multiple edges, the induced map
is a surjective graph homomorphism.
Therefore it suffices to show that
if $(u,v)$ is an  edge of $\G$ then $([u],[v])$ is an edge of $\G^c$.

Suppose then that $u,v\in X$, $u\neq v$ and $(u,v)$ is an edge of $\G$.
If $[u]=[v]$ and $|[u]|=1$ or $[u]=[u]_o$ then there are no edges of $\G$
joining
elements of $[u]$ to each other. Therefore if  $[u]=[v]$ we may assume
that $[u]=[v]=[u]_\perp$. In this case $[u]$ is a simplex, with more than
one element since $u\neq v$, and so there
is a loop $e$ in $\G^\tc$ from $[u]$ to itself. Thus $(u,v)$ maps to $e$,
as required.

Now suppose that $[u]\neq [v]$. If $|[u]|=|[v]|=1$ then $([u],[v])$ is clearly
an edge of $\G^\tc$. Suppose then that $|[v]|\ge 2$ and that $z\in [v]$, $z\neq v$.
Then $(u,v)\in E(\G)$ implies $u\in v^\perp$. As either $z\sim_\perp v$ or
$z\sim_o v$ and $z\neq v$ it follows that $z\in u^\perp$. If $|[u]|=1$ this implies
that $([u],[v])\in E(\G^\tc)$. If $|[u]|\ge 2$ then let $w\in [u]$, $w\neq u$. Then
$w\neq z$ (as $[u]\neq [v]$) and $z\in u^\perp$ implies $z\in w^\perp$. Hence $(w,z)\in
E(\G)$ and it follows that $([u],[v])$ is an edge of $\G^\tc$.
\end{proof}

As usual we extend $\tc$ to a map from subsets of $X$ to subsets
of $X^\tc$ by setting $\tc(Y)=\cup_{y\in Y} \{\tc(y)\}$, for $Y\subseteq X$.
If $[y]\in X^\tc$ then
$[y]^\perp =\{[u]\in X^\tc: d([u],[y])\le 1\}=\{[u]\in X^\tc: d(u,y)\le 1\}$,
by definition of $\G^\tc$, so for all $y\in X$,
\[
\tc(y)^\perp =\{[u]\in X^\tc:u\in y^\perp\}=
\bigcup_{u\in y^\perp}\{\tc(u)\}=\tc(y^\perp).
\]
Now suppose that $Z=\{z_1,\ldots ,z_n\}\subseteq X$. Then
$\tc(Z)^\perp=(\cup_{i=1}^n \tc(z_i))^\perp
=\cap_{i=1}^n\tc(z_i)^\perp
=\cap_{i=1}^n\tc(z_i^\perp)
$
.
Clearly $\cap_{i=1}^n\tc(z_i^\perp)\supseteq \tc(\cap_{i=1}^n z_i^\perp)$.
On the other hand, if $[u]\in \cap_{i=1}^n\tc(z_i^\perp)$ then
$[u]\in [z_i]^\perp$, so $d([u],[z_i])\le 1$ and so $d(u,z_i)\le 1$, for
$i=1,\ldots , n$. Therefore $u\in \cap_{i=1}^n z_i^\perp$ from which it follows that
$[u]\in \tc(\cap_{i=1}^n z_i^\perp)$. Hence
$\tc(Z)^\perp=\tc(\cap_{i=1}^n z_i^\perp)=\tc(Z^\perp)$.

Now restricting the map $\tc$ to closed sets we see that if
$Y\in \CS(\G)$ then $Y=Z^\perp$, for some $Z\subseteq X$ so
$\tc(Y)=\tc(Z^\perp)=\tc(Z)^\perp\in \CS(\G^\tc)$. Hence $\tc$ induces a map from
$\CS(\G)$ to $\CS(\G^\tc)$, which we denote by $\tc_L$.
Let $L$ denote the lattice $\CS(\G)$ and $L^\tc$ the lattice $\CS(\G^\tc)$.
\begin{prop}\label{prop:clat}
The map $\tc_L:L\maps L^c$ is  a lattice epimorphism which preserves the unary relation
$\perp$: that is $\tc_L(Y^\perp) =\tc_L(Y)^\perp$, for all $Y\in L$.
\end{prop}
\begin{proof}
As $\tc$ is a surjective map it follows that every subset of $X^\tc$ is the image
of a subset of $X$. If $W$ is a closed subset of $X^\tc$ then $W=V^\perp$ for
some subset $V$ of $X^\tc$. Choose $Y\subseteq X$ such that $\tc(Y)=V$. As we
have seen above we have  $\tc(Y^\perp)=V^\perp = W$. As $Y^\perp\in L$ we have
$\tc_L(Y^\perp)=W$, so $\tc_L$ is a surjective map. It therefore suffices to
show that $\tc_L$ is a lattice homomorphism. If $S, T\in L$ then $S=U^\perp$ and
$T=V^\perp$, for some $U,V\in L$. Then $S\wedge T=S\cap T$ and
\begin{align*}
\tc_L(S)\wedge \tc_L(T)
&=\tc_L(U^\perp)\cap \tc_L(V^\perp)\\
&=\tc_L(U)^\perp\cap \tc_L(V)^\perp\\
&=(\tc_L(U) \cup \tc_L(V))^\perp\\
&=(\tc_L(U\cup V))^\perp\\
&=\tc_L(U^\perp \cap V^\perp)\\
&=\tc_L(S\wedge T).
\end{align*}
Moreover
\begin{align*}
\tc_L(S\vee T)
&= \tc_L((S\cup T)^{\perp\perp})\\
&=(\tc_L(S\cup T))^{\perp\perp}\\
&=(\tc_L(S)\cup \tc_L(T))^{\perp\perp}\\
&=\tc_L(S)\vee \tc_L(T).
\end{align*}
Hence $\tc_L$ is a lattice homomorphism as claimed.
\end{proof}

We make $\G^\tc$ into a labelled graph as follows. For $x\in X$ define
$\mu(x)=|[x]|$ and $\nu(x)=1,$ if $x\in M_1$, $\nu(x)=\perp$, if $x\in M_\perp$ and
$\nu(x)=o$, if $x\in M_o$. Define a labelling function $l:X^c\maps \NN\times
\{1,\perp,o\}$ by $l([y]) =(\mu(y),\nu(y))$, for all $y\in X^\tc$.
\begin{expl}
In drawing the compressed graph vertices with labels of the form $(1, 1)$ or
$(r, \perp)$ are represented as single circles containing the integer $1$ or $r$,
respectively, and vertices with labels of the form $(r,o)$ are represented as
two
concentric circles containing the integer $r$, as in Figure \ref{fig:compgraph}.
\end{expl}
\begin{figure} 
\begin{center}
\includegraphics[scale=0.4]{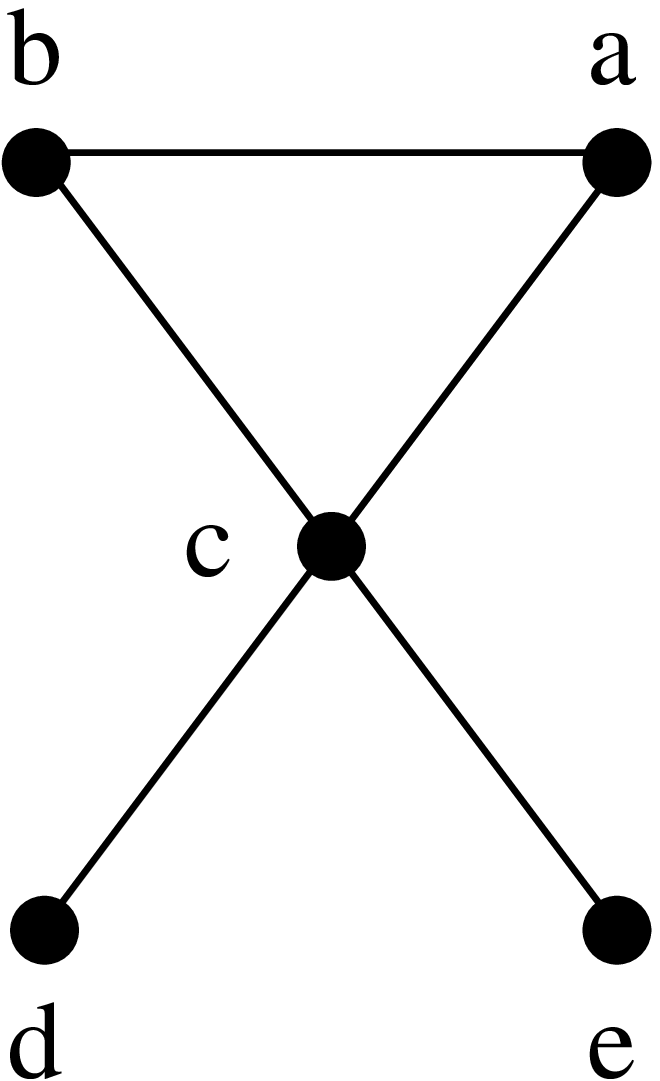}\quad\quad\quad\quad\quad\quad\includegraphics[scale=0.4]{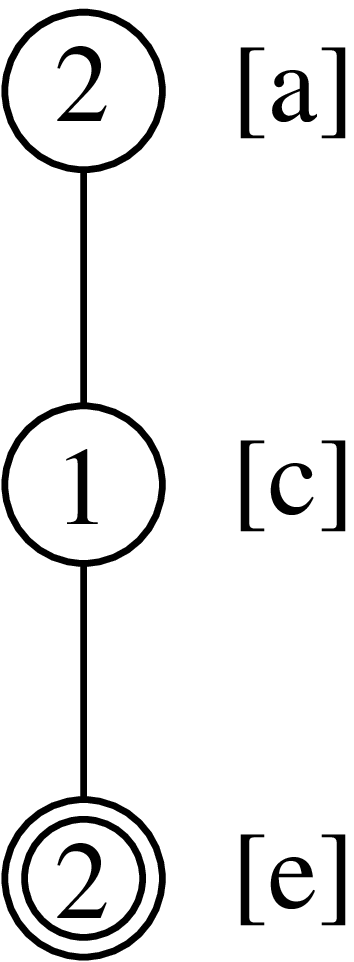}
\caption{A graph $\G$ and its compression $\G^c$}\label{fig:compgraph}
\end{center}
\end{figure}
Let $\Aut(\G^c)$ denote the group of automorphisms of $\G^c$ as a labelled graph: that is
$\phi \in \Aut(\G^c)$ if and only if $\phi $ is an automorphism of the graph $\G^\tc$
such that $l(\phi([v]))=l([v])$, for all $[v]\in X^\tc$.
Let $\Aut(\G)$ denote the group of graph automorphisms of $\G$ and let
 $\phi\in \Aut(\G)$. Since $\phi$ is an automorphism we have $\phi(u^\perp)
=\phi(u)^\perp$, for all $u\in X$. It follows that if $u,v\in X$ and $[u]=[v]$ then
$[\phi(u)]=[\phi(v)]$. Applying $\phi^{-1}$ to the latter equality we see that
$[u]=[v]$ if and only if $[\phi(u)]=[\phi(v)]$. Since $\tc$ and $\phi$ are
graph homomorphisms it follows that $\phi_\tc=\tc\circ \phi$ 
is an
automorphism of $\G^\tc$ as a labelled graph: 
that is $\phi_\tc\in \Aut(\G^\tc)$.
Denote by $\Aut(\tc)$ the map which takes $\phi \in \Aut(\G)$ to
$\phi_\tc\in \Aut(\G^c)$.
For $[v]\in X^\tc$ let $S_{\mu(v)}$ denote the symmetric group of degree $\mu(v)$.

\begin{prop}
The map $\Aut(\tc)$ is an epimorphism from $\Aut(\G)$ to $\Aut(\G^\tc)$.
There is a split short exact sequence
\begin{equation}\label{eq:exaut}
1\maps \prod_{[v]\in X^\tc}S_{\mu(v)}\maps \Aut(\G)\xrightarrow{\Aut(\tc)} \Aut(\G^c)\maps 1.
\end{equation}
\end{prop}
\begin{proof}
We have seen that $\Aut(\tc)$ is a map from $\Aut(\G)$ to $\Aut(\G^c)$.
If $\phi, \phi^\prime\in \Aut(\G)$ then
$(\phi\circ \phi^\prime)_\tc([v])=[\phi\circ\phi^\prime(v)]
=\phi_\tc([\phi^\prime(v)])=\phi_\tc\circ \phi_\tc^\prime([v])$,
for all $[v]\in X^\tc$.
Hence $\Aut(\tc)$ is a homomorphism.

Let $[v]\in X^c$ and consider the subgraph $\G([v])$ of $\G$. If
$\phi \in \Aut(\G([v]))$ then we may extend $\phi$ to $\G$ by setting
$\phi(u)=u$, for all $u\notin [v]$. Hence we may regard $\Aut(\G([v]))$
as a subgroup of $\Aut(\G)$. If $u,v\in X$ and $[u]\neq [v]$ then
$\phi\circ\phi^\prime=\phi^\prime\circ \phi$, for all $\phi\in \Aut(\G([u]))$
and $\phi^\prime\in \Aut(\G([v]))$. Moreover, as $[u]\cap [v]=\emptyset$
we have $\Aut(\G([u]))\cap \Aut(\G([v]))=1$. Therefore
$\Aut(\G)$ contains the subgroup
$
A=\prod_{[v]\in X^c}\Aut(\G([v])).
$
If $\phi\in A$ then $\phi(v)\in [v]$, for all $v\in [v]$ and for
all $[v]\in X^\tc$. Therefore $\phi\in \ker(\Aut(\tc))$ and so
$A\subseteq \ker(\Aut(\tc))$. Conversely if $\phi \in \ker(\Aut(\tc))$
then $[\phi(v)]=[v]$ so $\phi(v)\in [v]$, for all $v\in X$. Hence
if $\phi \in \ker(\Aut(\tc))$ then $\phi|_{[v]}\in \Aut(\G([v]))$ and
so $\phi \in A$. Therefore $A=\ker(\Aut(\tc))$. For all $[v]\in X^c$ the
graph $\G([v])$ is either a simplex or a free co-simplex so
$\Aut(\G([v]))$ is isomorphic to the symmetric group
$S_{\mu(v)}$
of degree $\mu(v)$. Therefore $\prod_{[v]\in X^c}S_{\mu(v)}\cong A$.
To show that the sequence \eqref{eq:exaut}
is exact it remains only to show that
$\Aut(\tc)$ is surjective. However we shall first
construct an embedding $\i:\Aut(\G^\tc)\maps \Aut(\G)$.

Fix a transversal $V=\{v_1,\ldots, v_n\}$ for the map $\tc:\G\maps
\G^c$. For $i$ such that $1\le i\le n$ choose an ordering
$(v_{i,1},\ldots v_{i,\mu(v_i)})$ of the class $[v_i]$, with 
$v_i=v_{i,1}$.
Then $X=\sqcup_{i=1}^n\sqcup_{j=1}^{\mu(v_i)}\{v_{i,j}\}$.
For $i,k$ such that
$1\le i\le k\le n$ and $l(v_i)=l(v_k)$,
define a map $\t_{i,k}:[v_i]\maps [v_k]$ by $\t_{i,k}(v_{i,j})=v_{k,j}$,
$j=1,\ldots, \mu(v_i)$. Note that, as  $l(v_i)=l(v_k)$ the map
$\t_{i,k}$ is a graph isomorphism from $\G([v_i])$ to $\G([v_k])$.
If $\t_{i,k}$ is defined and $i<k$ we define $\t_{k,i}=\t_{i,k}^{-1}$.
Furthermore
if $\t_{i,k}$ and $\t_{k,l}$ are both defined then so is $\t_{i,l}$ and
by construction $\t_{i,l}=\t_{k,l}\circ \t_{i,k}$.

Now let $\phi_\tc\in \Aut(\G^\tc)$ and define a map $\phi$ of $X$ to
itself as follows. Let $v\in X$. Then
$[v]=[v_i]$, so $v=v_{i,j}$, for unique $i$ and $j$. There is a unique
$k$ such that $\phi_\tc([v_i])=[v_k]$ and
as $l(\phi_c(v_i))=l(\phi_c(v_i))$ the map $\t_{i,k}$ is defined.
Set $\phi(v)=\t_{i,k}(v_{i,j})
=v_{k,j}$. As all the $\t_{i,k}$ are isomorphisms and as $\phi_c$ is
a graph automorphism it follows that $\phi$ is a graph automorphism. Thus
$\i:\phi_c\maps \phi$   is a map from $\Aut(\G^c)$ to $\Aut(\G)$.
That $\i$ is an injective homomorphism follows directly from the
definition.

If $\phi_c\in \Aut(\G^\tc)$ and $[v]\in X^\tc$ then
$\Aut(\tc)\circ \i (\phi_c)$ maps $[v]$ to $[\i\phi_c(v)]=\phi_c([v])$,
so $\Aut(\tc)\circ \i$ is the identity on $\Aut(\G^c)$.
This implies that $\Aut(\tc)$ is surjective; so the sequence
\eqref{eq:exaut}
is exact.
Furthermore
$\i$ is a transversal for $\Aut(\tc)$ and
so \eqref{eq:exaut}
splits, as claimed.
\end{proof}
The compression $\G^c$ of $\G$ gives rise to
a natural decomposition of $G(\G)$ which we now describe;
using the following  generalisation
of a partially commutative group. Let
$\G$ be a graph and to each vertex of $\G$ associate a group
$G_v$. Let $F=\ast_{v\in V(\G)}G_v$ and let $N$ be the normal
subgroup of $F$ generated by all elements of the form
$[g_u,g_v]$, where $g_u\in G_u$, $g_v\in G_v$ and $u$ and $v$ are
joined by an edge of $\G$. The group
$G=F/N$ is called a  {\em partially commutative product
 of groups}. If all the vertex groups $G_v$ are infinite cyclic groups
then $G$ is a partially commutative group.
In the case in question
take $\G^c$ to be the underlying graph and associate
the the partially commutative group with commutation graph $\G([v])$ to the vertex $[v]$.
The vertex groups are all then free Abelian
groups or free groups.

\end{document}